\documentclass[draftcls,onecolumn]{IEEEtran}

\usepackage{cite}
\usepackage[cmex10]{amsmath}
\usepackage{url}
\usepackage{graphicx}
\usepackage{accents}
\usepackage{multirow}
\usepackage{subfig}

\graphicspath{{figures/}}

\usepackage{color}
\usepackage{ifthen}
\usepackage{amssymb}
\usepackage{amsthm}

\DeclareMathOperator{\spn}{span}

\newtheorem{theorem}{Theorem}

\newtheorem*{theorem*}{Theorem}
\newtheorem{proposition}{Proposition}

\newtheorem{definition}{Definition}

\makeatletter
\newcommand{\bdiv}{%
  \nonscript\mskip-\medmuskip\mkern5mu%
  \mathbin{\operator@font div}\penalty900\mkern5mu%
  \nonscript\mskip-\medmuskip
}
\makeatother

\definecolor{darkgreen}{rgb}{0,0.5,0}
\newboolean{showcomments}
\setboolean{showcomments}{true}
\newcommand{\desmond}[1]{\ifthenelse{\boolean{showcomments}}
{ \textcolor{red}{(Desmond says:  #1)}}{}}
\newcommand{\adam}[1]{\ifthenelse{\boolean{showcomments}}
{ \textcolor{red}{(Adam says:  #1)}}{}}
\newcommand{\enrique}[1]{\ifthenelse{\boolean{showcomments}}
{ \textcolor{red}{(Enrique says:  #1)}}{}}

\newcommand{\hide}[1]{}

\newcommand{\ubar}[1]{\underaccent{\bar}{#1}}

\begin{document}
\title{Distributed optimization decomposition for joint economic dispatch and frequency regulation}

\author{
Desmond Cai$^*$, Enrique Mallada$^\dagger$ and Adam Wierman$^{\dagger\dagger}$%
\thanks{$^*$Desmond Cai is with Department of Electrical Engineering,
        California Institute of Technology. Email:
        {\tt\small wccai@caltech.edu}.}%
\thanks{$^{\dagger}$Enrique Mallada is with the Department of Electrical and Computer Engineering, Johns Hopkins University. Email: {\tt\small mallada@jhu.edu}.}
\thanks{$^{\dagger\dagger}$Adam Wierman is with the Department of Computing and Mathematical Sciences,
        California Institute of Technology. Email:
        {\tt\small adamw@caltech.edu}.}
\thanks{This work was supported
by ARPA-E grant DE-AR0000226, Los Alamos National Lab through
an DoE grant DE-AC52-06NA25396, DTRA through grant HDTRA
1-15-1-0003, Skoltech, NSF grant 1545096 as part of the
NSF/DHS/DOT/NASA/NIH Cyber-Physical Systems Program, NSF grant
NETS-1518941, and NSF grant EPAS-1307794.  A preliminary and abridged version has appeared in~\cite{CaiMalladaWierman2015a}.}%
\thanks{The authors would like to thank Ben Hobbs from John Hopkins and Steven H. Low from Caltech for insightful discussions.}
}


\maketitle

\begin{abstract}
Economic dispatch and frequency regulation are typically viewed as fundamentally different problems in power systems and, hence, are typically studied separately. In this paper, we frame and study a joint problem that co-optimizes both slow timescale economic dispatch resources and fast timescale frequency regulation resources. We show how the joint problem can be decomposed without loss of optimality into slow and fast timescale sub-problems that have appealing interpretations as the economic dispatch and frequency regulation problems respectively. We solve the fast timescale sub-problem using a distributed frequency control algorithm that preserves the stability of the network during transients. We solve the slow timescale sub-problem using an efficient market mechanism that coordinates with the fast timescale sub-problem. We investigate the performance of the decomposition on the IEEE 24-bus reliability test system.
\end{abstract}


%
\IEEEpeerreviewmaketitle

\section{Introduction}

One of the major objectives of every Independent System Operator (ISO) is to schedule generation to meet demand at every time instant~\cite{WoodWollenberg1996, BergenVittal2000, MachowskiBialek2008}. This is a challenging task -- it involves responding rapidly to supply-demand imbalances, minimizing generation costs, and respecting operating limitations (such as ramp constraints, capacity constraints, and line constraints). Due to the complexity of this global system operation problem, it is typically divided into two separate problems: \emph{economic dispatch}, which focuses on control of slower timescale resources and is solved using market mechanisms, and \emph{frequency regulation}, which focuses on control of faster timescale resources and is solved using engineered controllers. Economic dispatch and frequency regulation are typically studied independently of each other.

\emph{Economic dispatch} operates at the timescale of 5 minutes or longer and focuses on cost efficiency. In particular, the economic dispatch problem seeks to optimally schedule generators to minimize total generation costs. Economic dispatch has a long history~\cite{WoodWollenberg1996,carpentier1979optimal,kirschen2004front,baldick2004theory,caisoOptimization,schweppe1988spot}. It is currently implemented using a market mechanism known as supply function bidding. In this mechanism, generators submit supply functions to the ISO which specify (as a function of price) the quantity a generator is willing to produce. The ISO solves a centralized optimization problem (over single or multiple time periods) to schedule generators to minimize system costs while satisfying demand and slow timescale operating constraints (such as line constraints, capacity constraints, ramping constraints, security constraints, etc.). Each generator is compensated at the locational marginal price (LMP) which reflect the system cost of serving an incremental unit of demand at its node.

\emph{Frequency regulation} operates at a faster timescale (from a few minutes to 30 seconds) and focuses on stability rather than efficiency. In particular, the ISO seeks to restore the nominal frequency in the system by rescheduling fast ramping generators. Frequency regulation has a long history~\cite{IbraheeKkumarKothari2005,BergenVittal2000,deMello:1973jy}. It is currently implemented by a mechanism known as Automatic Generation Control (AGC). In this mechanism, the ISO computes the aggregate generation that would rebalance power within each independent control area (and hence restore nominal frequency) and allocates the imbalance generation among generators based on the solution of the previous economic dispatch run~\cite{WoodWollenberg1996}. These allocations determine the setpoints in a distributed control algorithm that drives the power system to a stable operating point using local information on frequency deviations. The generators are compensated at the LMP from the previous economic dispatch run.



\subsection{Contributions of this paper}

While economic dispatch and frequency regulation each have large and active literatures; these literatures are almost completely disparate. While there have been studies on integrating the two mechanisms more efficiently~\cite{ThatteZhangXie2011}, we are not aware of any rigorous analysis of whether their combination solves the global system operator's goal of dispatching generation resources efficiently across both timescales. The goal of this paper is to initiate such a study.


Our main result provides an initial answer. In the context of a DC power flow model and two classes of generators (dispatch and regulation), we show that the global system operator's problem can be decomposed into two sub-problems that correspond to the economic dispatch and frequency regulation timescales, without loss of optimality, as long as the ISO is able to estimate the difference between the average LMP in the frequency regulation periods and the LMP in the economic dispatch period (Theorem \ref{thm:decomposition}). This result can be viewed as a first-principles justification for the existing separation of power systems control into economic dispatch and frequency regulation problems. Moreover, this result provides a guide to modify the existing architecture to \emph{optimally} control power systems across timescales. In particular, using this result, we design an optimal control policy for frequency regulation and an optimal market mechanism for economic dispatch, in a way such that the control and market mechanisms jointly solve the global system operator's problem. Our mechanims differ from existing economic dispatch and frequency regulation mechanisms in important ways.



In the case of frequency regulation (Section \ref{sec:frequency}), our mechanism has a key advantage over the AGC mechanism in that our mechanism is efficient. The frequency regulation controller proposed in this paper is built on the distributed controller in \cite{ZhaoMalladaLowBialek2016,2014arXiv1410.2931M} and controls generation based on information about generators' costs in a way such that the power system converges to an operating point that minimizes system costs. On the other hand, AGC allocates generation based on participation factors, which might not reflect actual costs, and hence the resulting allocation might not be efficient. In \cite{LiChenZhao2014}, the authors proposed a modification of the participation factors so that the AGC mechanism is cost efficient. However, unlike our mechanism, the mechanism in \cite{LiChenZhao2014} does not respect line constraints.



In the case of economic dispatch (Section~\ref{sec:market}), our mechanism has a key advantage over the existing economic dispatch operations in that it coordinates efficiently with the frequency regulation timescale. This coordination does not require additional communication in the market beyond the existing mechanism used in practice. This coordination involves two main components. First, our economic dispatch mechanism communicates the supply function bids from the generators to the frequency regulation mechanism, which uses them in the distributed controllers to allocate frequency regulation resources efficiently. In contrast, the AGC mechanism allocates frequency regulation resources without regard to generation costs. Second, our economic dispatch mechanism accounts for the value that economic dispatch resources provide to frequency regulation. It does so by adjusting the resource costs in the economic dispatch objective based on the difference between the LMP in the frequency regulation periods and that in the economic dispatch period. In contrast, the existing economic dispatch objective does not perform this adjustment and hence might allocate economic dispatch resources inefficiently.

In practice, the ISO is unlikely to be able to estimate exactly the adjustment it should make to the economic dispatch objective. In Section~\ref{sec:simulation}, we investigate numerically the sensitivity of the suboptimality of our decomposition to those estimation errors on the IEEE 24-bus reliability test system.

\section{System Model}

Our aim is to understand how the combination of economic dispatch and frequency regulation can dispatch generation resources efficiently across both timescales. To this end, we formulate a model of the global objective that includes balancing supply and demand at both timescales. We use a DC power flow model and consider two generation types -- dispatch and regulation -- which differ in responsiveness.

Consider a connected network consisting of a set of nodes $N$ and a set of links $L$. We focus on a single economic dispatch interval of the real-time market which is typically 5 minutes in existing markets. We partition this time interval into $K$ discrete periods numbered $1,\ldots,K$. In general, the length of each period may range from as little as seconds to as long as minutes. However, in this work, we focus on the case where each period is on the order tens of seconds.

\subsection{Stochastic demand}

We use a stochastic demand model motivated by the frameworks in~\cite{CarpentierGohenCulioliRenaud1996,TakritiBirgeLong1996,OzturkMazumdarNorman2004}. Assume that there is a set of possible demand outcomes $S$ that can be described by a scenario tree (an example is given in \figurename~\ref{fig:scenario-tree}). For each outcome $s\in S$, let $d_{s,n}\in\mathbb{R}$ denote the real power demand at node $n\in N$ and $\mathbf{d}_s := (d_{s,n}, n \in N) \in \mathbb{R}^N$ denote the vector of demands at all nodes. In addition, let $\kappa(s)\in\left\{1,\ldots,K\right\}$ denote the period of this outcome and $p_s$ denote the probability of this outcome conditioned on the information that the period is $\kappa(s)$. Hence, $\sum_{\left\{s|\kappa(s)=k\right\}} p_s = 1$ for each $k\in\left\{1,\ldots,K\right\}$. Without loss of generality, we assume that $\kappa(1) = 1$ and $p_1 = 1$. That is, there exists an outcome labeled $1 \in S$ associated with period $1$ and the demand in that period is deterministic.

\begin{figure}
\centering
\includegraphics[width=0.33\textwidth]{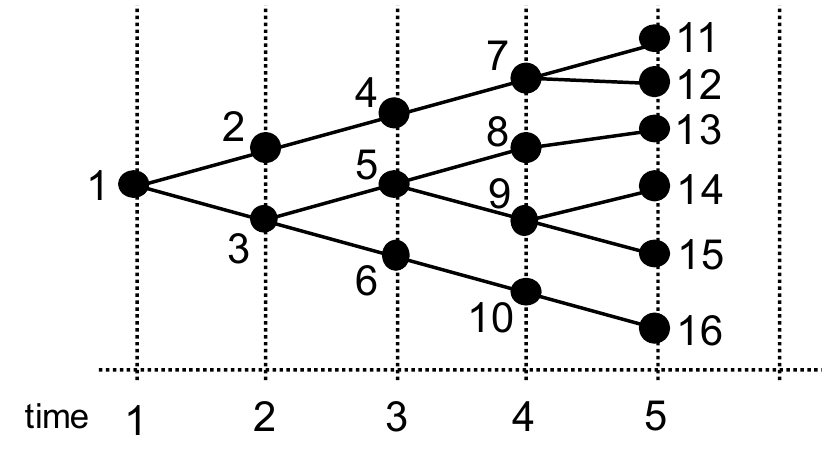}
\caption{Example of a scenario tree with $S=16$ outcomes over $K=5$ periods. The outcomes are numbered $1,\ldots,S$.}
\label{fig:scenario-tree}
\end{figure}

\subsection{Generation}

We assume that each node $n\in N$ has two generators -- a dispatch generator and a regulation generator -- where the regulation generator is more responsive than the dispatch generator. To model the differing responsiveness, we assume that the dispatch generator produces at a constant level over the entire economic dispatch interval while the regulation generator may change its production level every period after uncertain demand is realized~\cite{ElaOMalley2012}. Formally, we assume that the dispatch generator produces $q_n^b \in \mathbb{R}$ in all outcomes, and the regulation generator produces $q_n^p \in \mathbb{R}$ in period $1$ and $q_n^p + r_{s,n}^p \in \mathbb{R}$ in each subsequent outcome $s \in S\setminus\{1\}$. Hence, $q_n^p$ and $r_{s,n}^p$ can be interpreted as the regulation generator's setpoint and recourse respectively. To simplify notations, we define a dummy variable $r_{1,n}^p := 0$ so that we may write the regulation generator's production in period $1$ as $q_n^p + r_{1,n}^p$. We assume that the regulation and dispatch generators have capacity constraints $[\ubar{q}_n^p,\bar{q}_n^p]$ and $[\ubar{q}_n^b,\bar{q}_n^b]$ respectively, and incur costs $c_n^p(q_{n}^p+r_{s,n}^p)$ and $c_n^b(q_{n}^b)$ respectively in period $\kappa(s)$, where the functions $c_n^p\!:\![\ubar{q}_n^p,\bar{q}_n^p] \!\rightarrow\! \mathbb{R}_{+}$ and $c_n^b\!:\![\ubar{q}_n^b,\bar{q}_n^b]\!\rightarrow\!\mathbb{R}_{+}$ are strictly convex and continuously differentiable. 

Define vectors $\mathbf{q}^p \!:=\! (q_{n}^p, n\in N)$, $\mathbf{r}_s^p \!:=\! (r_{s,n}^p, n\in N)$, $\mathbf{q}^b \!:=\! (q_{n}^b, n\in N)$,  $\ubar{\mathbf{q}}^p \!:=\! (\ubar{q}_n^p, n\in N)$, $\ubar{\mathbf{q}}^b\!:=\! (\ubar{q}_n^b, n\in N)$, $\bar{\mathbf{q}}^p \!:=\! (\bar{q}_n^p, n\in N)$, $\bar{\mathbf{q}}^b\!:=\! (\bar{q}_n^b, n\in N)$. Then the generation constraints in outcome $s\in S$ are given by:
\begin{align}
\ubar{\mathbf{q}}^b & \leq \mathbf{q}^b \leq \bar{\mathbf{q}}^b,
\label{eq:baseload-constraint}
\\
\ubar{\mathbf{q}}^p & \leq \mathbf{q}^p + \mathbf{r}_s^p \leq \bar{\mathbf{q}}^p.
\label{eq:peaker-constraint}
\end{align}
We also let the vector $\mathbf{r}^p := (\mathbf{r}^p_s, s\in S)$.

\subsection{Network constraints}
\label{sec:model-network-constraints}

Note that $\mathbf{q}^b + \mathbf{q}^p + \mathbf{r}_s^p - \mathbf{d}_s$ is the vector of nodal injections for $s \in S$. Thus, the supply-demand balance constraint is:
\begin{align}
\mathbf{1}^\top ( \mathbf{q}^b + \mathbf{q}^p + \mathbf{r}_s^p - \mathbf{d}_{s} ) = 0,
\label{eq:supply-demand-constraint}
\end{align}
where $\mathbf{1}\in\mathbb{R}^N$ denotes the vector of all ones.

We adopt the DC power flow model for line flows. Let $\theta_{s,n}$ denote the phase angle of node $n$. Without loss of generality, assign each link $l$ an arbitrary orientation and let $i(l)$ and $j(l)$ denote the tail and head of the link respectively. Let $B_l$ denote the sensitivity of the flow with respect to changes in the phase difference $\theta_{s,i(l)}-\theta_{s,j(l)}$ and let $v_{s,l}$ denote its power flow. Define the vectors $\boldsymbol\theta_s := (\theta_{s,n}, n\in N)$ and $\mathbf{v}_s := (v_{s,l}, l\in L)$ and the matrix $\mathbf{B}:=\text{diag}(B_l, l\in L)$. Then, the line flows are given by
$\mathbf{v}_s = \mathbf{B}\mathbf{C}^\top \boldsymbol\theta_s$
where $\mathbf{C}\in\mathbb{R}^{N \times L}$ is the incidence matrix of the directed graph. And the injections are:
\begin{align}
\mathbf{q}^b + \mathbf{q}^p + \mathbf{r}_s^p - \mathbf{d}_s = \mathbf{C} \mathbf{v}_s = \mathbf{L}\boldsymbol\theta_s,
\label{eq:supply-demand-constraint-2}
\end{align}
where $\mathbf{L} := \mathbf{C}\mathbf{B}\mathbf{C}^\top$.

Note that~\eqref{eq:supply-demand-constraint} and~\eqref{eq:supply-demand-constraint-2} are equivalent. For any set of injections that satisfy~\eqref{eq:supply-demand-constraint}, we can always find $\boldsymbol\theta_s$ that satisfies~\eqref{eq:supply-demand-constraint-2}. Conversely, since $\mathbf{1}^\top \mathbf{C} = 0$, any injections that satisfy~\eqref{eq:supply-demand-constraint-2} also satisfy~\eqref{eq:supply-demand-constraint}. Hence, the line flows can be written in terms of the power injections:
\begin{equation*}
\mathbf{v}_s = \mathbf{B}\mathbf{C}^\top \mathbf{L}^\dagger (\mathbf{q}^b + \mathbf{q}^p + \mathbf{r}_s^p - \mathbf{d}_s),
\end{equation*}
where $\mathbf{L}^\dagger$ denotes the pseudo-inverse of $\mathbf{L}$. Let $\mathbf{H}:= \mathbf{B}\mathbf{C}^\top \mathbf{L}^\dagger$. Let $f_l$ denote the capacity of line $l$ and define the vector $\mathbf{f} := (f_l, l\in L)$. Then the line flow constraints are:
\begin{align}
-\mathbf{f} \leq \mathbf{H}\left(\mathbf{q}^b + \mathbf{q}^p + \mathbf{r}_s^p - \mathbf{d}_s \right) \leq \mathbf{f}.
\label{eq:line-flow-constraint}
\end{align}
To simplify notations, we define the set $\Omega(\mathbf{d}_s)$ of feasible generation for a given demand vector $\mathbf{d}_s$ as:
\begin{align*}
\Omega(\mathbf{d}_s)
:=
\left\{(\mathbf{q}^b,\mathbf{q}^p,\mathbf{r}_s^p)  :  \eqref{eq:baseload-constraint}, \eqref{eq:peaker-constraint}, \eqref{eq:supply-demand-constraint}, \eqref{eq:line-flow-constraint} \; \text{holds} \right\}.
\end{align*}

\subsection{System operator's objective}

The global system operator's objective is to allocate the dispatch and regulation generations $(\mathbf{q}^b, \mathbf{q}^p, \mathbf{r}^p)$ to minimize the expected cost of satisfying demand and operating constraints. This is formalized as follows.
\begin{equation*}
\begin{array}{rl}
SYSTEM:\underset{{\mathbf{q}}^b,{\mathbf{q}}^p,{\mathbf{r}}^p}{\mathrm{min}} &\!\! \displaystyle\sum_{s\in S} p_s \sum_{n\in N} \left( c_n^b(q_{n}^b) + c_n^p(q_{n}^p + r_{s,n}^p) \right)
\\
\mathrm{s.t.}\;\;\; &\!\! (\mathbf{q}^b,\mathbf{q}^p,\mathbf{r}_s^p) \in \Omega(\mathbf{d}_s), \quad \forall s\in S,
\\
&\!\! \mathbf{r}_1^p = \mathbf{0}.
\end{array}
\end{equation*}
We assume that this optimization is feasible. Note that $SYSTEM$ differs from the existing economic dispatch mechanism which minimize costs under the assumption that the demand during all the $K$ periods in the economic dispatch interval is equal to the demand ${\mathbf{d}}_1$ in period $1$.

Let $\lambda_s$ and $(\ubar{\boldsymbol\mu}_s,\bar{\boldsymbol\mu}_s)$ be the Lagrange multipliers associated with constraints~\eqref{eq:supply-demand-constraint} and~\eqref{eq:line-flow-constraint} respectively in $SYSTEM$. Then, the function $\boldsymbol\pi:\mathbb{R}\times \mathbb{R}^{2L}_+ \rightarrow \mathbb{R}^N$, defined by:
\begin{equation}
\boldsymbol\pi (\lambda_s, \ubar{\boldsymbol\mu}_s, \bar{\boldsymbol\mu}_s)
:=
\lambda_s\mathbf{1} + \mathbf{H}^\top (\ubar{\boldsymbol\mu}_s - \bar{\boldsymbol\mu}_s),
\label{eq:price}
\end{equation}
gives the nodal prices in outcome $s \in S$.


\section{Architectural Decomposition}
\label{sec:decomposition}

Our main result is a decomposition of $SYSTEM$ into setpoint and recourse sub-problems. Importantly, our decomposition identifies a rigorous connection between the setpoint and recourse sub-problems that ensures that the combination solves $SYSTEM$. 
In particular, our decomposition divides $SYSTEM$ into sub-problems $ED$ and $FR$ defined by:
\begin{align*}
&\begin{array}{rl}
ED(\mathbf{d}_1):\quad \underset{{\mathbf{q}}^b,{\mathbf{q}}^p}{\mathrm{min}}\! & \displaystyle\sum_{n\in N} \left( K c_n^b(q_{n}^b) + K c_n^p(q_{n}^p) - \delta_n q_n^b\right)
\\
\mathrm{s.t.} & (\mathbf{q}^b,\mathbf{q}^p,\mathbf{0}) \in \Omega(\mathbf{d}_1),
\end{array}
\\
&\begin{array}{rl}
FR(\mathbf{q}^{b},\mathbf{q}^p,\mathbf{d}_s):\quad \underset{{\mathbf{r}}^p_s}{\mathrm{min}} & \displaystyle\sum_{n\in N} c_n^p(q_{n}^p+r_{s,n}^p)
\\
\mathrm{s.t.} & (\mathbf{q}^{b},\mathbf{q}^p,\mathbf{r}_{s}^p) \in \Omega(\mathbf{d}_s),
\end{array}
\end{align*}
where $\boldsymbol\delta\in\mathbb{R}^N$ is a constant. $ED({\mathbf{d}}_1)$ is implemented in time period $1$ and $FR(\mathbf{q}^{b},\mathbf{q}^p,\mathbf{d}_s)$ is implemented in subsequent time periods $\kappa(s)>1$.

We denote the first optimization problem by $ED$, since it optimizes only generation setpoints $(\mathbf{q}^b,\mathbf{q}^p)$ assuming constant demand ${\mathbf{d}}_1$ over the $K$ time periods, and hence it is on the same timescale as the existing economic dispatch mechanism. We denote the second optimization problem by $FR$, since it optimizes regulation generators' recourse production $\mathbf{r}_s^p$ in subsequent time periods, and hence it is on the same timescale as the existing frequency regulation mechanism.

\begin{definition}
We say that $SYSTEM$ can be optimally decomposed into $ED$-$FR$ if $(\mathbf{q}^{b},\mathbf{q}^{p},\mathbf{r}^p)$ is an optimal solution to $SYSTEM$ if and only if $\mathbf{r}_1^p=\mathbf{0}$, $(\mathbf{q}^{b},\mathbf{q}^{p})$ is an optimal solution to $ED(\mathbf{d}_1)$, and $\mathbf{r}_s^{p}$ is an optimal solution to $FR(\mathbf{q}^{b},\mathbf{q}^p,\mathbf{d}_s)$ for all $s\in S$.
\end{definition}

\begin{theorem}[Decomposition]
\label{thm:decomposition}
Let $\lambda_s$ and $(\ubar{\boldsymbol\mu}_s, \bar{\boldsymbol\mu}_s)$ be any Lagrange multipliers associated with constraints~\eqref{eq:supply-demand-constraint} and~\eqref{eq:line-flow-constraint} respectively in $SYSTEM$.
\begin{itemize}

\item[(a)] If $\boldsymbol\delta$ is the average, over all time periods, of the difference between the expected nodal prices in each period and that in period~$1$, that is, for each $n\in N$,
\begin{align}
\delta_n
=
\sum_{s\in S} p_s \left(\pi_n(\lambda_s,\ubar{\boldsymbol\mu}_s, \bar{\boldsymbol\mu}_s) - \pi_n(\lambda_1,\ubar{\boldsymbol\mu}_1,\bar{\boldsymbol\mu}_1)\right),
\label{eq:optimal-delta}
\end{align}
then $SYSTEM$ can be optimally decomposed into $ED$-$FR$.

\item[(b)] If $SYSTEM$ can be optimally decomposed into $ED$-$FR$, then for all $n$ such that $\ubar{q}_n^b < q_n^b < \bar{q}_n^b$ and $\ubar{q}_n^p < q_{1,n}^p < \bar{q}_n^p$,~\eqref{eq:optimal-delta} holds. 
\end{itemize}
\end{theorem}

The proof of Theorem~\ref{thm:decomposition} is given in the Appendix. The result follows from analyzing the Karush-Kuhn-Tucker conditions of the system operator's problem and those of $ED$ and $FR$.
As mentioned, we denote the two sub-problems by $ED$ and $FR$ because they focus on the economic dispatch and frequency regulation timescales respectively. Hence, these sub-problems can serve as guides for the optimal design of economic dispatch and frequency regulation mechanisms. The insights are immediate in the case of economic dispatch and we show how $ED$ leads to an improved market mechanism in Section~\ref{sec:market}. However, the insights may not be as clear in the case of frequency regulation. We show in Section~\ref{sec:frequency} that $FR$ can in fact be solved via distributed frequency control algorithms, although these algorithms deviate from current practice that do not optimize generation costs.

The most important feature of Theorem~\ref{thm:decomposition} is that, one way to choose generation setpoints optimally at the economic dispatch timescale, is to include, in the optimization objective, an offset of the dispatch generators' marginal costs by the expected changes in nodal prices during the frequency regulation timescale. The latter can be interpreted as the expected changes in the marginal value of dispatch generation. Hence, if the latter is zero, then generation setpoints can be chosen optimally at the economic dispatch timescale without regard to the behavior of the system in the frequency regulation timescale~\cite{CaiMalladaWierman2015a}.

An important extension of this result is to understand the efficiency loss of the decomposition when we are unable to estimate the RHS of \eqref{eq:optimal-delta} accurately. Note that negative estimation errors cause $ED(\mathbf{d}_1)$ to use less than optimal dispatch resources (and more than optimal regulation resources) and vice versa. We investigate the efficiency loss in Section~\ref{sec:simulation}. In such situations, the dispatch generation $\mathbf{q}^b$ might not be optimal, and therefore $FR(\mathbf{q}^b,\mathbf{q}^p,\mathbf{d}_s)$ might not be feasible. To ensure that $FR(\mathbf{q}^b,\mathbf{q}^p,\mathbf{d}^s)$ is feasible, we may modify $ED(\mathbf{d}_1)$ into a robust optimization problem by adding constraints $(\mathbf{q}^b,\mathbf{q}^p,\mathbf{r}_s^p)\in\Omega(\mathbf{d}_s)$ for all $s\in S\setminus\{1\}$. The size of such a problem is exponential in $S$ but can be reduced using the technique in~\cite{Minoux2011}. Note that this should not be viewed as a drawback of our decomposition, as the current practice
based on AGC might also not be feasible.
In practice, the risks of infeasibility are mitigated using reserves. Moreover, our decomposition has the advantage that it coordinates the economic dispatch and frequency regulation resources efficiently, and hence, may reduce reserve requirements.

Theorem~\ref{thm:decomposition} is close in spirit to work in communication networks that use optimization decomposition to justify and optimize protocol layering~\cite{chiang2007layering,palomar2006tutorial,cai2012optimal}. Hence, \emph{Theorem~\ref{thm:decomposition} provides a rigorous way to think about architectural design of power networks.}


\section{Distributed Frequency Regulation}
\label{sec:frequency}

 This section illustrates how to implement the solution to $FR$ using distributed frequency regulation controllers. 
Besides achieving optimality, a practical implementation should 
preserve network stability, be robust to unexpected system events, aggregate network information in a distributed manner, and satisfy constraints~\eqref{eq:peaker-constraint}, \eqref{eq:supply-demand-constraint} and~\eqref{eq:line-flow-constraint}. The distributed algorithm that we provide in this section satisfies all the above characteristics. It can be interpreted as performing distributed frequency regulation by sending different regulation signals to each bus.


\subsection{Dynamic model}
\label{ssec:dynamic-model}

Before introducing our algorithm we add dynamics to our system model to describe the system behavior within a single time period. Let $t$ denote the time evolution within the time period of outcome $s$, and assume without loss of generality that $t\in(k,k+1]$ where $k=\kappa(s)$. Let ${\mathbf{r}}_{s}^p(t) := (r_{s,n}^p(t), n\in N)$ denote the recourse quantities generated by the regulation generators at time $t$. We assume that dispatch generation and demand do not change within the time period.

Then, the system changes within the time period are governed by the swing equations which we assume to be:
\begin{subequations}
\label{eq:baseload-dynamics}
\begin{align}
\dot {\boldsymbol\theta}_s(t) =\; & {\boldsymbol\omega}_s(t);\\
\mathbf{M}\dot{\boldsymbol\omega}_s(t) =\; & \mathbf{q}^b+ \mathbf{q}^p + \mathbf{r}_s^p(t) -\mathbf{d}_{s}  - \mathbf{D}{\boldsymbol\omega}_s(t) - \mathbf{L} {\boldsymbol\theta}_s(t),
\end{align}
\end{subequations}
where ${\boldsymbol\omega}_s(t):=(\omega_{s,n}(t), n\in N)$ are the frequency deviations from the nominal value at time $t$, $\boldsymbol\theta_s(t) := (\theta_{s,n}(t), n\in N)$ are the phase angles at time $t$, ${\mathbf{M}} := \text{diag}(M_1,\ldots,M_N)$ where $M_n$ is the aggregate inertia of the generators at node $n$, and ${\mathbf{D}} := \text{diag}(D_1,\ldots,D_N)$ where $D_n$ is the aggregate damping of the generators at node $n$. The notation $\dot x$ denotes the time derivative, i.e. $\dot x = dx/dt$.
Equation~\eqref{eq:baseload-dynamics} is a linearized version of the nonlinear network dynamics~\cite{BergenHill1981,BergenVittal2000}, and has been widely used in the design of frequency regulation controllers. See, e.g., \cite{deMello:1973ea,deMello:1973jy}.

\subsection{Distributed frequency regulation}
\label{ssec:distr-FR}


We now introduce a distributed, continuous-time algorithm that provably solves $FR$ \emph{while preserving system stability}.
Our solution is based on a novel reverse and forward engineering approach for distributed control design in power systems~\cite{ZhaoTopcuLiLow2014,zhao2015acc,2014arXiv1410.2931M,LiChenZhao2014,YouChen2014ReverseAndForward,Jokic2009}.
The algorithm operates as follows. Each regulation generator $n$ updates its power generation using
\begin{equation}\label{eq:control-signal}
r_{s,n}^p(t) = [c_n^{p\prime -1}(-\omega_{s,n}(t) - \pi_{s,n}^p(t))]_{\ubar{q}_n^p-q_n^p}^{\bar q_n^p-q_n^p},
\end{equation}
where $c_n^{p\prime}(x)=\frac{\partial}{\partial x}c_n^p(x)$ and $c_n^{p \prime -1}$ denotes its inverse. The projection $[r]_{\ubar{q}_n^p-q_n^p}^{\bar q_n^p-q_n^p}$ ensures that ${\ubar{q}_n^p}-q_n^p \leq r \leq \bar q_n^p-q_n^p$ (or equivalently ${\ubar{q}_n^p} \leq r+q_n^p \leq \bar q_n^p$) and $\pi_{s,n}^p(t)$ is a control signal generated using:
\begin{subequations}\label{eq:distr-freq-regulation}
\begin{align}
\hspace{-9pt}DFR: \; \dot{\boldsymbol\pi}_s^p(t)
&=
\displaystyle\boldsymbol\zeta^\pi \left( {\mathbf{q}}^{b}+{\mathbf{q}}^{p} + {\mathbf{r}}_s^p(t)-{\mathbf{d}}_{s}   - \boldsymbol L\boldsymbol\phi_s(t) \right); \label{eq:distr-freq-regulation-a}
\\[2pt]
\dot{\bar{\boldsymbol\mu}}_s(t)
&=
\displaystyle\boldsymbol\zeta^{\bar{\mu}}\big[{\mathbf{BC}}^\top\boldsymbol\phi_s(t) - {\mathbf{f}}\big]^+_{\bar{\boldsymbol\mu}_s}; \label{eq:distr-freq-regulation-b}
\\
\ubar{\dot{\boldsymbol\mu}}_s(t)
&=
\displaystyle\boldsymbol\zeta^{\ubar{\mu}}\big[- {\mathbf{f}} -{\mathbf{BC}}^\top\boldsymbol\phi_s(t) \big]^+_{\ubar{\boldsymbol\mu}_s}; \label{eq:distr-freq-regulation-c}
\\
\dot{\boldsymbol\phi}_s(t)
&=
\boldsymbol\chi^\phi\left({\mathbf{L}}\boldsymbol\pi_s^p(t) - {\mathbf{CB}}(\bar{\boldsymbol\mu}_s(t) -  \ubar{\boldsymbol\mu}_s(t))\right), \label{eq:distr-freq-regulation-d}
\end{align}
\end{subequations}
where $\boldsymbol\zeta^\pi:=\text{diag}(\zeta^\pi_1,\dots,\zeta^\pi_N)$, $\boldsymbol\zeta^{\bar\mu}:=\text{diag}(\zeta^{\bar{\mu}}_1,\dots,\zeta^{\bar{\mu}}_L)$, $\boldsymbol\zeta^{\ubar\mu}:=\text{diag}(\zeta^{\ubar{\mu}}_1,\dots,\zeta^{\ubar{\mu}}_L)$, $\boldsymbol\chi^\phi:=\text{diag}(\chi^\phi_1,\dots,\chi^\phi_N)$ denote the respective control gains. The element-wise projection $[\mathbf{y}]^+_{\mathbf{x}}:=([y_n]^+_{x_n}, n\in N)$ ensures that the dynamics $\dot{\mathbf{x}}= [\mathbf{y}]^+_{\mathbf{x}}$ have a solution $\mathbf{x}(t)$ that remains in the positive orthant, that is, $[y_n]^+_{x_n}=0$ if $x_n=0$ and $y_n<0$, and $[y_n]^+_{x_n}=y_n$ otherwise.

The proposed solution \eqref{eq:control-signal} -- \eqref{eq:distr-freq-regulation} can be interpreted as a frequency regulation algorithm in which each regulation generator receives a different regulation signal \eqref{eq:control-signal} depending on its location in the network. The key step in the design of $DFR$ is reformulating $FR$ into the following equivalent optimization problem:
\begin{subequations}
\begin{align}
FR'(\boldsymbol q^b,\boldsymbol q^p, \boldsymbol d_s):\nonumber\\
\quad \underset{{\mathbf{r}}_s^p,{\boldsymbol \omega}_s,{\mathbf{v}}_s,{\boldsymbol \phi}_s}{\mathrm{min}} \quad& \sum_{n\in N} \left(c_n^p(q_{n}^p+r_{s,n}^p) + D_n{\omega_{s,n}^2}/{2}\right)
\nonumber \\
\mathrm{s.t.} \qquad\;
& \mathbf{q}^{b}+\mathbf{q}^{p} + \mathbf{r}_s^p - \mathbf{d}_s - \mathbf{D}\boldsymbol\omega_s = \mathbf{C} \mathbf{v}_s; \label{eq:FR'_k-b}
\\
&  {\mathbf{q}}^{b} +\mathbf{q}^{p}+ {\mathbf{r}}_s^p - {\mathbf{d}}_s = {\mathbf{L}} \boldsymbol\phi_s; \label{eq:FR'_k-c}
\\
 &-{\mathbf{f}} \leq {\mathbf{BC}}^\top \boldsymbol\phi_s \leq {\mathbf{f}};\label{eq:FR'_k-d}
\\
& \ubar{\mathbf{q}}^p \leq {\mathbf{q}}^p \leq \bar{\mathbf{q}}^p. \label{eq:FR'_k-e}
\end{align}
\end{subequations}
Recall from Section~\ref{sec:model-network-constraints} that $\mathbf{v}_s$ denote line flows. Constraint \eqref{eq:FR'_k-b} is reformulated from the per node supply-demand balance constraint~\eqref{eq:supply-demand-constraint-2}, and makes explicit the fact that, whenever supply and demand do not match, the mismatch is compensated by a change in the frequency. Constraint \eqref{eq:FR'_k-c} ensures that $\boldsymbol\omega_s=0$ at the optimal solution so that supply and demand are balanced. Constraint \eqref{eq:FR'_k-d} imposes line flow limits. However, instead of using actual line flows ${\mathbf{v}}_s$, these limits are imposed on {\textit{virtual flows}} ${\mathbf{BC}}^\top \boldsymbol\phi_s$, which are identical to line flows at the optimal solution~\cite{2014arXiv1410.2931M}.

It can be shown that $FR^\prime$ has a primal-dual algorithm that contains the component~\eqref{eq:baseload-dynamics} resembling power network dynamics and the components~\eqref{eq:control-signal}~--~\eqref{eq:distr-freq-regulation} that can be implemented via distributed communication and computation. This new problem $FR'$ also makes explicit the role of frequency in maintaining supply-demand balance.

The next proposition formally relates the optimal solutions of $FR$ and $FR'$ and guarantees the optimality of \eqref{eq:control-signal}~--~\eqref{eq:distr-freq-regulation}.

\begin{proposition}[Optimality]\label{prop:optimality}
Let ${\mathbf{r}}^{p}_{s}$ and $({{\mathbf{r}}_s^{p\prime}},\boldsymbol\omega_s',\mathbf{v}_s',\boldsymbol\phi_s')$ be optimal solutions of $FR$ and $FR'$ respectively. Then, the following statements are true: (i) {\textit{Frequency restoration:}} $\boldsymbol\omega_s'=0$; (ii) {\textit{Generation equivalence:}} ${\mathbf{r}}^{p}_{s}={{\mathbf{r}}^{p\prime}_s}$; (iii) {\textit{Line flow equivalence:}}  $ {\mathbf{H}} \left( {\mathbf{q}}^{b} + {\mathbf{q}}^{p}+{\mathbf{r}}_s^p - {\mathbf{d}}_{s} \right) = {\mathbf{BC}}^\top\boldsymbol\phi_s'$. Moreover,
there exists ${\boldsymbol \theta}_s' \in \mathbb{R}^N$ and ${\mathbf{y}}_s' \in \mathbb{R}^L$, satisfying ${\mathbf{Cy}}_s'=0$, such that
${\mathbf{v}}_s' = {\mathbf{BC}}^\top\boldsymbol\theta_s' +{\mathbf{y}}_s'$ and ${\mathbf{BC}}^\top\boldsymbol\phi_s' ={\mathbf{BC}}^\top\boldsymbol\theta_s'$.
And $({{\mathbf{r}}_s^{p\prime}},\boldsymbol\omega_s',\boldsymbol\theta_s',\boldsymbol\phi_s',\boldsymbol\pi_s^{p\prime},\ubar{\boldsymbol\mu}_s',\bar{\boldsymbol\mu}_s')$ is an equilibrium point of \eqref{eq:baseload-dynamics} -- \eqref{eq:distr-freq-regulation} if and only if
$({{\mathbf{r}}_s^{p\prime}},\boldsymbol\omega_s',{\mathbf{v}}_s',\boldsymbol\phi_s',\boldsymbol\pi_s^{p\prime},\ubar{\boldsymbol\mu}_s',\bar{\boldsymbol\mu}_s')$  is a primal-dual optimal solution of $FR'$, where $\boldsymbol \omega_s'$, $\boldsymbol\pi_s^{p\prime}$, and $(\ubar{\boldsymbol\mu}_s',\bar{\boldsymbol\mu}_s')$ are the Lagrange multipliers associated with constraints~\eqref{eq:FR'_k-b}, \eqref{eq:FR'_k-c}, and~\eqref{eq:FR'_k-d}, respectively.
\end{proposition}

The proof of Proposition~\ref{prop:optimality} is given in the Appendix. What remains is to guarantee the convergence of the distributed frequency regulation algorithm.  


 \begin{proposition}[Convergence]
 \label{prop:convergence}
If $c_n^p$ is twice continuous differentiable with ${c_n^{p\prime\prime}}\geq\alpha>0$ (i.e., $\alpha$-strictly convex) and $c_n^p(q_{n}^p+r_{s,n}^p)\!\rightarrow\! +\infty$ as $q_{n}^p+r_{s,n}^p\!\rightarrow\!\{\ubar{q}_n^p,\bar{q}_n^p\}$, then $\mathbf{r}_s^p(t)$ in \eqref{eq:baseload-dynamics}~--~\eqref{eq:distr-freq-regulation} converge globally to an optimal solution of $FR$.
 \end{proposition}

The proof of Proposition~\ref{prop:convergence} follows from \cite{2014arXiv1410.2931M} and uses the machinery developed in~\cite{Cherukuri201610} to handle projections~\eqref{eq:distr-freq-regulation-b}~--~\eqref{eq:distr-freq-regulation-c}. By substituting the line flows ${\mathbf{v}}_s(t) = {\mathbf{BC}}^\top\boldsymbol\theta_s(t)$ into~\eqref{eq:baseload-dynamics} and eliminating $\boldsymbol\theta_s(t)$, we can show that
 the entire system~\eqref{eq:baseload-dynamics}~--~\eqref{eq:distr-freq-regulation} is a primal-dual algorithm of $FR'$ (see \cite[Theorem 5]{2014arXiv1410.2931M}). Therefore, Theorem 10 in \cite{2014arXiv1410.2931M} guarantees global asymptotic convergence to an equilibrium point which by Proposition \ref{prop:optimality} is an optimal solution of both $FR'$ and $FR$. Our setup is simpler than the controllers in \cite{2014arXiv1410.2931M}, which had additional states, but the same proof technique applies. Although Proposition~\ref{prop:convergence} requires costs to blow up as regulation generations approach minimum and maximum capacities, this assumption is not restrictive, as it can be achieved by adding a barrier function to the actual cost before implementing in the controllers.

\section{Market Mechanism for Economic Dispatch}
\label{sec:market}

This section illustrates how to implement the solution to $ED$ through a market mechanism for economic dispatch. The mechanism works in the following manner. In the first time period, the ISO collects supply function bids from generators (both dispatch and regulation) and uses those bids to solve $ED$. Then, in subsequent time periods, the ISO uses the regulation generators' supply function bids to implement the controller in~\eqref{eq:control-signal}. This mechanism is efficient if $SYSTEM$ can be decomposed into $ED$-$FR$ and does not require any more communication than the existing market mechanisms used in practice.

\subsection{Market model}

We assume that generators are price-takers. Let $\pi_{n}^b$ denote the price paid to dispatch generator $n$ in each period and $\pi_{s,n}^p$ denote the price paid to regulation generator $n$ in outcome $s$. Then, the expected profit of the dispatch and regulation generators at node $n$ are:
\begin{align*}
&\mathrm{PF}^b_n(q_n^b, \pi_{n}^b)
:=
K\left(\pi_{n}^b q_n^b - c_n^b(q_n^b)\right),
\\
&\mathrm{PF}^p_n((q_n^p + r_{s,n}^p, \pi_{s,n}^p), s\in S)
\\
&:=
\displaystyle\sum_{s\in S} p_s \left( \pi_{s,n}^p \left(q_{n}^p + r_{s,n}^p\right) - c_n^p(q_{n}^p+r_{s,n}^p)\right).
\end{align*}
Note that the regulation generator's profit is a function of its total production $q_n^p + r_{s,n}^p$ in each outcome $s\in S$. The supply function bids indicate the quantities the generators are willing to produce at every price. We assume that these bids are chosen from a parameterized family of functions. In particular, for node $n$, we represent the dispatch and regulation generators' supply functions by parameters $\alpha_n^b > 0$ and $\alpha_n^p > 0$ respectively, and these bids indicate that the dispatch generator is willing to supply the quantity $q_n^b = [\alpha_n^b s_n^{b}(\pi_{n}^b)]_{\ubar{q}_n^b}^{\bar{q}_n^b}$ in the first time period and the regulation generator is willing to supply the quantity $q_{n}^p + r_{s,n}^p = [\alpha_n^p s_n^{p}(\pi_{s,n}^p)]_{\ubar{q}_n^p}^{\bar{q}_n^p}$ in outcome $s$, for some fixed functions $s_n^b:[\ubar{q}_n^b,\bar{q}_n^b]\rightarrow\mathbb{R}_+$ and $s_n^p:[\ubar{q}_n^p,\bar{q}_n^p]\rightarrow\mathbb{R}_+$.\footnote{Numerous studies have explored different functional forms of the supply functions and their impact on market efficiency, e.g., see~\cite{rudkevich2005supply,baldick2002electricity,johari2011parameterized,baldick2004theory,klemperer1989supply}. The focus of this work is on illustrating that $ED$ can be implemented using a simple market mechanism. Hence, we restrict ourselves to linearly parameterized supply functions and leave the analyses of other more sophisticated supply functions to future work. We refer the reader to~\cite{johari2011parameterized} for some appealing properties of linearly parameterized supply functions.}
We also assume that $s_n^b(\pi_n^b) \neq 0$ for all $\pi_n^b\in\mathbb{R}$ and $s_n^p(\pi_{s,n}^p) \neq 0$ for all $\pi_{s,n}^p \in\mathbb{R}$.\footnote{This assumption is a technical condition to avoid the degenerate situation where a generator's supply quantity is not sensitive to its bid parameter which would occur if $s_n^{b}(\pi_n^b) = 0$ or $s_n^{p}(\pi_{s,n}^p) = 0$.} 
The generators choose their bids to maximize their profits subject to their capacity constraints. Note that the regulation generator submits only one supply function for all possible outcomes. Hence, its bid in the economic dispatch timescale is also used as its bid in the frequency regulation timescale.

The system operator interprets bids $\alpha_n^b$ and $\alpha_n^p$ as signals that the dispatch and regulation generators at node $n$ have marginal costs $\pi_{n}^b$ and $\pi_{s,n}^p$ respectively when supplying quantities $\alpha_n^b s_n^{b}(\pi_{n}^b)$ and $\alpha_n^p s_n^{p}(\pi_{s,n}^p)$ respectively. Hence, it associates with the generators the following bid cost functions:
\begin{align}
\hat{c}_n^b(q_n^b)
&:=
\int_{\ubar{q}_n^b}^{q_n^b} (s_n^{b})^{-1}(w/\alpha_n^b)\, dw,
\label{eq:baseload-bid-cost}
\\
\hat{c}_n^p(q_n^p)
&:=
\int_{\ubar{q}_n^p}^{q_n^p} (s_n^{p})^{-1}(w/\alpha_n^p)\, dw.
\label{eq:peaker-bid-cost}
\end{align}
Let $\boldsymbol\alpha^b := (\alpha_n^b, n\in N)$ and $\boldsymbol\alpha^p := (\alpha_n^p, n\in N)$ denote the vectors of bids. Given bids $(\boldsymbol\alpha^b, \boldsymbol\alpha^p)$, the system operator solves $ED$ to minimize expected bid costs. The prices for the regulation generator in the first time period are the nodal prices in $ED$ while the prices for the dispatch generator are the nodal prices offset by $\boldsymbol\delta$. Then, in each subsequent outcome $s\in S$, the system operator implements the controller in~\eqref{eq:control-signal} using regulation generators' bid costs. The prices are the nodal prices in $FR$ (which are computed by $DFR$).

\subsection{Market equilibrium}

Our focus is on understanding the efficiency of the mechanism. Formally, we consider the following notion of a competitive equilibrium.
\begin{definition}
\label{def:equilibrium}
We say that bids $(\boldsymbol\alpha^b,\boldsymbol\alpha^p)$ are a competitive equilibrium if there exists prices $\boldsymbol\pi^b \in \mathbb{R}^N$ and $\boldsymbol\pi^p = (\boldsymbol\pi_s^p, s\in S) \in \mathbb{R}^{NS}$ such that:
\begin{itemize}
\item[(a)] For all $n$, $\alpha_n^b$ is an optimal solution to:
\begin{align*}
\begin{array}{rl}
\displaystyle\max_{\hat{\alpha}_n^b>0} & \mathrm{PF}^b_n\left( [\hat{\alpha}_n^b s_n^{b}(\pi_{n}^b)]_{\ubar{q}_n^b}^{\bar{q}_n^b}, \pi_{n}^b\right).
\end{array}
\end{align*}

\item[(b)] For all $n$, $\alpha_n^p$ is an optimal solution to:
\begin{align*}
\begin{array}{rl}
\displaystyle\max_{\hat{\alpha}_n^p>0} & \mathrm{PF}^p_n\left(([\hat{\alpha}_n^p s_n^{p}(\pi_{s,n}^p)]_{\ubar{q}_n^p}^{\bar{q}_n^p}, \pi_{s,n}^p), s\in S\right).
\end{array}
\end{align*}

\item[(c)] $\boldsymbol\pi^b = (1/K)\left(\boldsymbol\pi({\lambda}_1,\ubar{\boldsymbol\mu}_1,\bar{\boldsymbol\mu}_1) + \boldsymbol\delta\right)$ and $\boldsymbol\pi_1^p = (1/K)\boldsymbol\pi({\lambda}_1,\ubar{\boldsymbol\mu}_1,\bar{\boldsymbol\mu}_1)$ where ${\lambda}_1$ and $(\ubar{\boldsymbol\mu}_1,\bar{\boldsymbol\mu}_1)$ are the Lagrange multipliers associated with constraints~\eqref{eq:supply-demand-constraint} and~\eqref{eq:line-flow-constraint} respectively in:
\begin{align*}
\begin{array}{rl}
\hat{ED}(\mathbf{d}_1): \quad \displaystyle\min_{\mathbf{q}^b,\mathbf{q}^p} & \displaystyle\sum_{n\in N} \left(K\hat{c}_n^b(q_{n}^b) + K\hat{c}_n^p(q_{n}^p) - \delta_n q_n^b\right) \notag
\\
\mathrm{s.t.}  & (\mathbf{q}^b,\mathbf{q}^p,\mathbf{0}) \in \Omega(\mathbf{d}_1).
\end{array}
\end{align*}

\item[(d)] For all $s \in S$, $\boldsymbol\pi_s^p = \boldsymbol\pi(\lambda_s,\ubar{\boldsymbol\mu}_s,\bar{\boldsymbol\mu}_s)$ where $\lambda_s$ and $(\ubar{\boldsymbol\mu}_s,\bar{\boldsymbol\mu}_s)$ are the Lagrange multipliers associated with constraints~\eqref{eq:supply-demand-constraint} and~\eqref{eq:line-flow-constraint} respectively in:
\begin{align*}
\begin{array}{rl}
\hspace{-37pt}\hat{FR}(\mathbf{q}^b,\mathbf{q}^p,\mathbf{d}_s): \quad \displaystyle\min_{\mathbf{r}_s^p} & \displaystyle\sum_{n\in N} \hat{c}_n^p(q_{n}^p+r_{s,n}^p) \notag
\\
\mathrm{s.t.}  & (\mathbf{q}^b,\mathbf{q}^p,\mathbf{r}_s^p) \in \Omega(\mathbf{d}_s),
\end{array}
\end{align*}
where $\mathbf{q}^b=\left([\alpha_n^b s_n^{b}(\pi_{n}^b )]_{\ubar{q}_n^b}^{\bar{q}_n^b},n\in N\right)$ and $\mathbf{q}^p=\left([\alpha_n^p s_n^{p}(\pi_{1,n}^p )]_{\ubar{q}_n^p}^{\bar{q}_n^p},n\in N\right)$.
\end{itemize}
At each node $n\in N$, the dispatch and regulation generators produce at setpoints $[\alpha_n^b s_n^{b}(\pi_{n}^b)]_{\ubar{q}_n^b}^{\bar{q}_n^b}$ and $[\alpha_n^p s_n^{p}(\pi_{1,n}^p)]_{\ubar{q}_n^p}^{\bar{q}_n^p}$ respectively in period $1$, and the regulation generator produces an additional quantity $[\alpha_n^p s_n^p(\pi_{s,n}^p)]_{\ubar{q}_n^p}^{\bar{q}_n^p} - [\alpha_n^p s_n^{p}(\pi_{1,n}^p)]_{\ubar{q}_n^p}^{\bar{q}_n^p}$ in outcome $s\in S$.
\end{definition}

The following is our main result for this section. It highlights that, as a consequence of Theorem~\ref{thm:decomposition}, any competitive equilibrium is efficient.
\begin{proposition}[Efficiency]
\label{prop:market}
Suppose that, for each $n\in N$, the functions $s_n^b(\cdot) =  c_n^{b\prime -1}(\cdot)/\gamma_n^b$ and $s_n^p(\cdot) = c_n^{p\prime -1}(\cdot)/\gamma_n^p$ for some constants $\gamma_n^b,\gamma_n^p > 0$. Let $\lambda_s$ and $(\ubar{\boldsymbol\mu}_s, \bar{\boldsymbol\mu}_s)$ be the Lagrange multipliers associated with constraints~\eqref{eq:supply-demand-constraint} and~\eqref{eq:line-flow-constraint} respectively in $SYSTEM$. Suppose that \eqref{eq:optimal-delta} holds. Then:
\begin{itemize}
\item[(a)] Any competitive equilibrium has a production schedule that solves $SYSTEM$.
\item[(b)] Any production schedule that solves $SYSTEM$ can be sustained by a competitive equilibrium.
\end{itemize}
\end{proposition}

Proposition~\ref{prop:market} resembles classical welfare theorems, e.g.,~\cite{johari2011parameterized,mas1995microeconomic,johari2009efficiency,wang2012dynamic}. However, it differs from typical competitive equilibria frameworks because each regulation generator is restricted to bidding a single supply function over the entire economic dispatch interval even though there are multiple fast timescale instances. The latter creates challenges in guaranteeing existence and efficiency of equilibria that do not arise in typical competitive equilibria frameworks. In particular, the space of bid functions needs to be sufficiently expressive for generators to convey their costs over multiple fast timescale instances via a single bid function. This is not an issue in market frameworks where separate bids are collected for separate market instances. Proposition~\ref{prop:market} circumvented this challenge by restricting supply functions to be in the linear space of regulation generators' true cost functions. An important extension is to understand the existence and efficiency of equilibria under less restrictive bid spaces. Proposition~\ref{prop:market} also highlights that nodal pricing is not always efficient and that the pricing mechanism needs to be jointly designed and analyzed with decomposition principles in order to achieve efficiency.

\section{Case Study}
\label{sec:simulation}

\begin{table*}[!t]
\caption{Generators on test system.}
\label{tbl:sim-gen-assignment}
\centering
\begin{tabular}{|c|c|c|c|c|}
\hline
Unit Group & Unit Type & Production & Marginal Cost & Assignment
\\
& & Range (MW) & Range ($\$$/MWh) &
\\ \hline
U12 & Oil/Steam & $[10,60]$ & $[58.14,64.446]$ & Dispatch
\\
U20 & Oil/CT & $[64,80]$ & $[130.0,130.0]$ & Regulation
\\
U50 & Hydro & $[60,300]$ & $[0.001,0.001]$ & Regulation
\\
U76 & Coal/Steam & $[60,304]$ & $[16.511,18.231]$ & Dispatch
\\
U100 & Oil/Steam & $[75,300]$ & $[46.295,54.196]$ & Dispatch
\\
U155 & Coal/Steam & $[216,620]$ & $[13.294,14.974]$ & Dispatch
\\
U197 & Oil/Steam & $[207,591]$ & $[49.57,51.405]$ & Dispatch
\\
U350 & Coal/Steam & $[140,350]$ & $[13.22,15.276]$ & Dispatch
\\
U400 & Nuclear & $[200,800]$ & $[4.466,4.594]$ & Dispatch
\\ \hline
\end{tabular}
\end{table*}

The efficiency of the mechanisms in Sections \ref{sec:frequency} and \ref{sec:market} depends on how accurately the system operator can predict the RHS of equation~\eqref{eq:optimal-delta}. In this section, we investigate the sensitivity of the performance of the decomposition to the value of $\boldsymbol\delta$ using a case study on the IEEE 24-bus reliability test system~\cite{WongAlbrechtAllanBillintonChenFongHaddadLiMukerjiPattonothers1999}.


Table~\ref{tbl:sim-gen-assignment} summarizes the properties of the generators on the system. We assume that the hydro and combustion turbine (CT) generators are regulation resources while all other generators are dispatch resources. Note that, the hydro resources, which generate between 60 to 300 MW, have the lowest marginal cost, while the CT resources, which generate between 64 to 80 MW, have the highest marginal cost.

We assume that there are $K=20$ time periods in the economic dispatch interval. Hence, each time period lasts 15 seconds. We construct the scenario tree as follows. Abusing notation, let $\mathbf{d}_{k}\in\mathbb{R}^N$ denote the demand at all nodes in period $k$. Set $\mathbf{d}_1$ to the values in the test system data and let
\begin{align*}
\mathbf{d}_{k}
=
\text{diag} \left(\mathbf{1} + \sum_{k'=1}^{k-1}\mathbf{w}_{k'} \right)\cdot \mathbf{d}_{1},
\end{align*}
where $\mathbf{w}_{k'}\sim \mathcal{N}(\mu_d\mathbf{1}, (0.002^2)\mathbf{I})$ is a Gaussian vector with mean $\mu_d\mathbf{1}$ and covariance $(0.002^2)\mathbf{I}$. We simulate $\mu_d = -0.0002, 0, +0.0002$ to model scenarios with increasing, constant, and decreasing demands, respectively. For each value of $\mu_d$, we generate 50 length-K samples of the random process and assign equal probabilities to all the samples. Hence, the scenario tree is a tall tree, where the root node has 50 children, and all other nodes either have one child or is a leaf node. \figurename~\ref{fig:demand-sample} shows the sample trajectories of total system demand. The RHS of equation~\eqref{eq:optimal-delta} have values $\delta_n^* = -80.34, 49.67, 49.72$ (in~\$/MWh) corresponding to $\mu_d = -0.0002, 0, +0.0002$, respectively. Note that the optimal $\delta_n$ is non-zero even when the demand evolution has zero mean. To study the impact when $\boldsymbol\delta$ deviates from $\boldsymbol\delta^*$, we consider
\begin{align*}
\boldsymbol\delta = \boldsymbol\delta^* + \boldsymbol\epsilon,
\end{align*}
where $\boldsymbol\epsilon\sim \mathcal{N}(\mu_\epsilon\mathbf{1}, \sigma_{\epsilon}^2\mathbf{I})$.  Hence, $\mu_\epsilon$ and $\sigma_{\epsilon}$ can be interpreted as the bias and standard deviation of the prediction errors. Given any specified $\mu_\epsilon$ and $\sigma_\epsilon$, we generate 50 samples of $\boldsymbol\delta$ and randomly match these samples to the 50 demand samples. \figurename~\ref{fig:inc-total-cost} shows the percentage increase in average total costs under the decomposition (compared to the optimal solution) for different values of $\mu_\epsilon$ and $\sigma_\epsilon$.

\begin{figure*}[!h]
\centering
\begin{minipage}[t]{\textwidth}
\centering
\subfloat[Increasing Demand]{\includegraphics[width=0.28\textwidth]{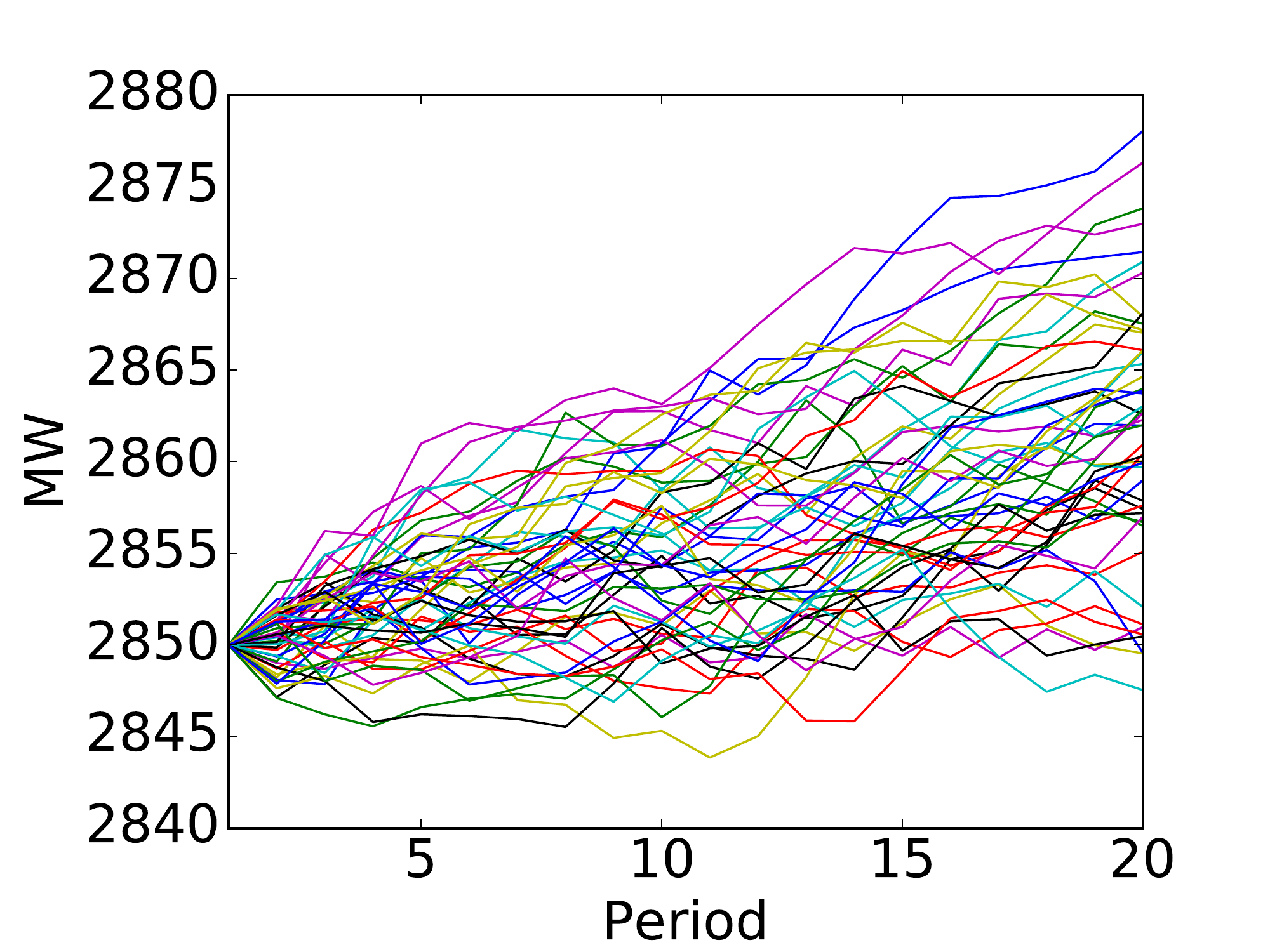}}
\subfloat[Constant Demand]{\includegraphics[width=0.28\textwidth]{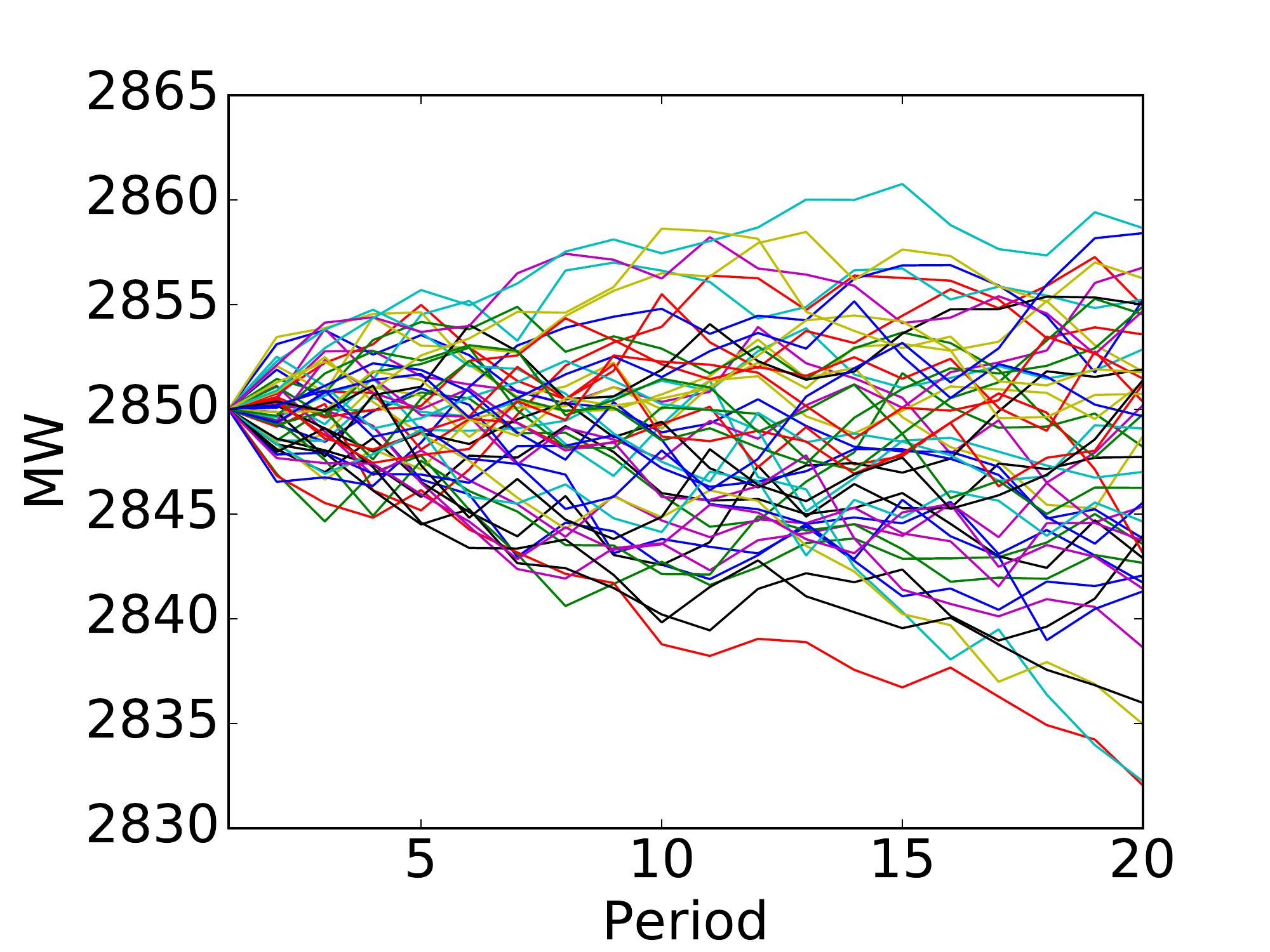}}
\subfloat[Decreasing Demand]{\includegraphics[width=0.28\textwidth]{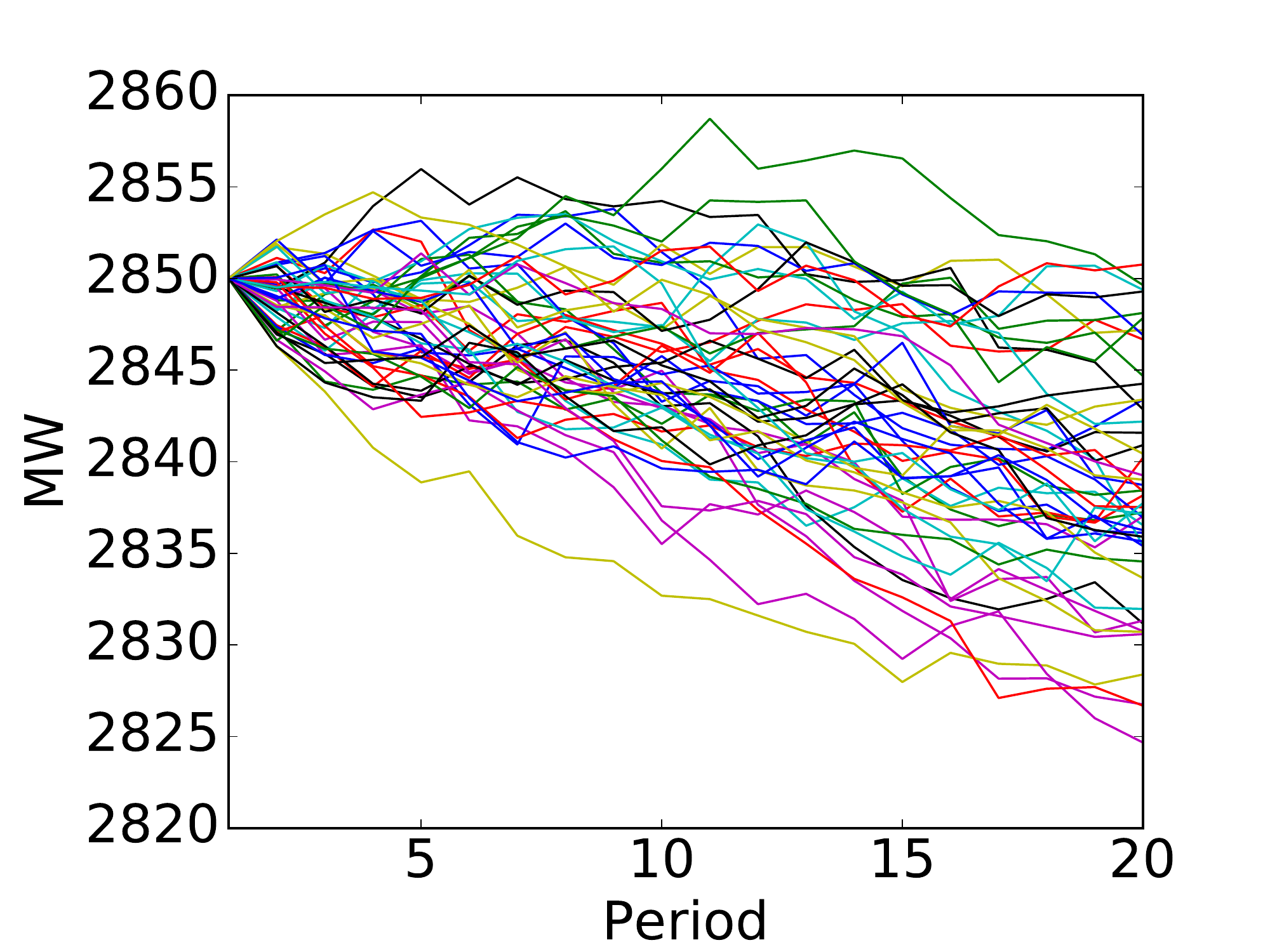}}
\caption{Samples paths of total demand.}
\label{fig:demand-sample}
\end{minipage}
\begin{minipage}[t]{\textwidth}
\centering
\subfloat[Increasing Demand]{\includegraphics[width=0.28\textwidth]{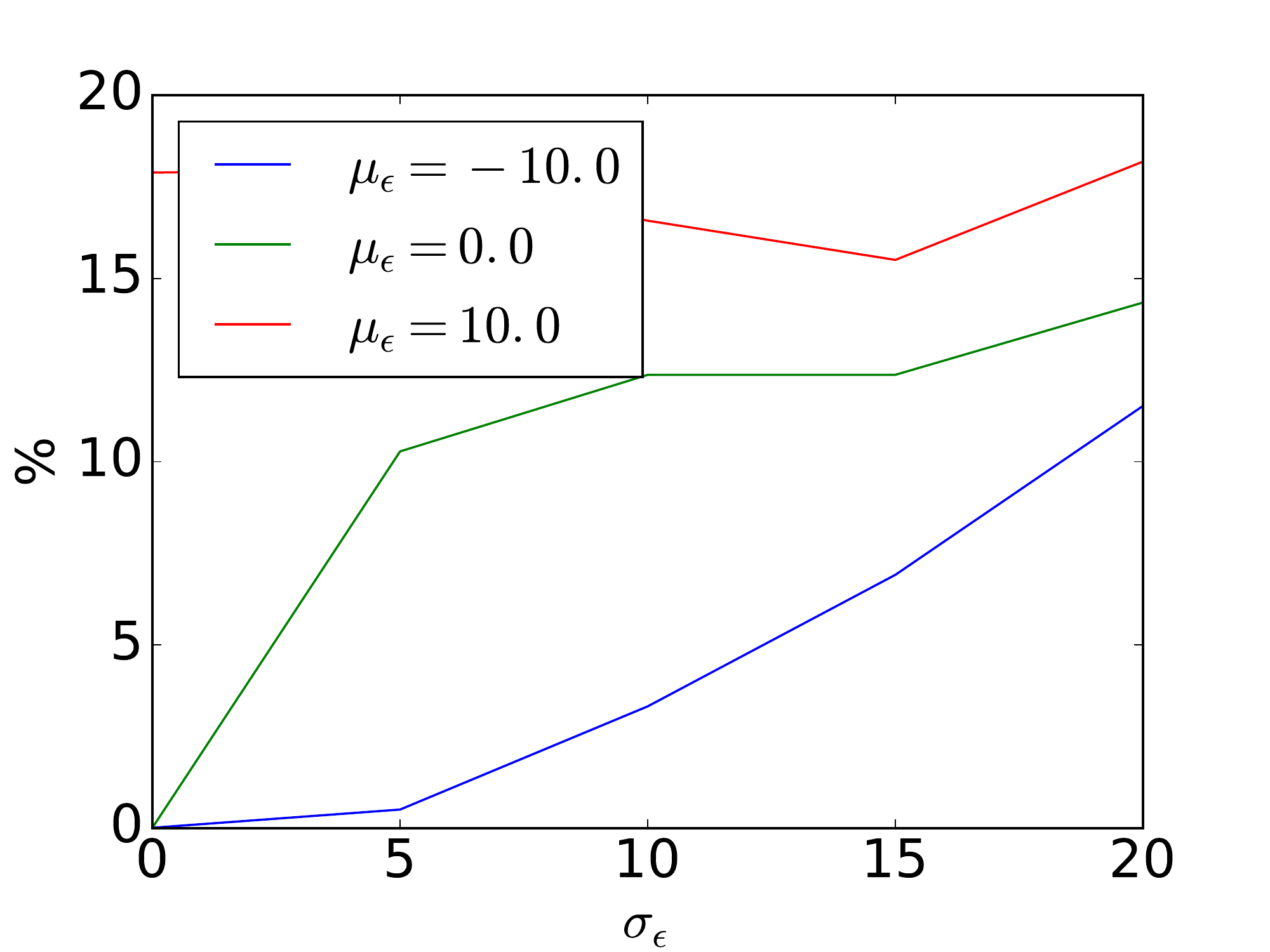}}
\subfloat[Constant Demand]{\includegraphics[width=0.28\textwidth]{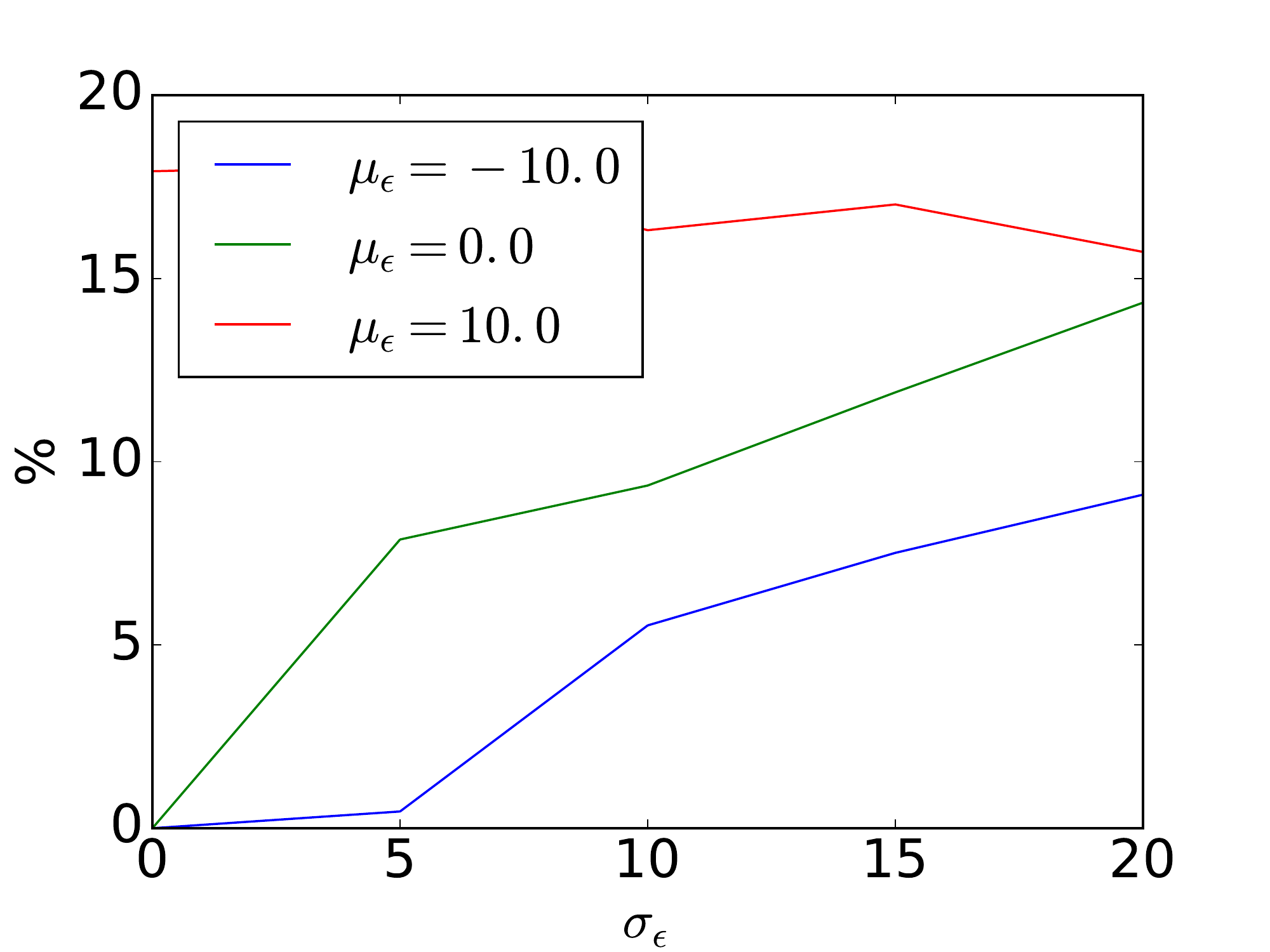}}
\subfloat[Decreasing Demand]{\includegraphics[width=0.28\textwidth]{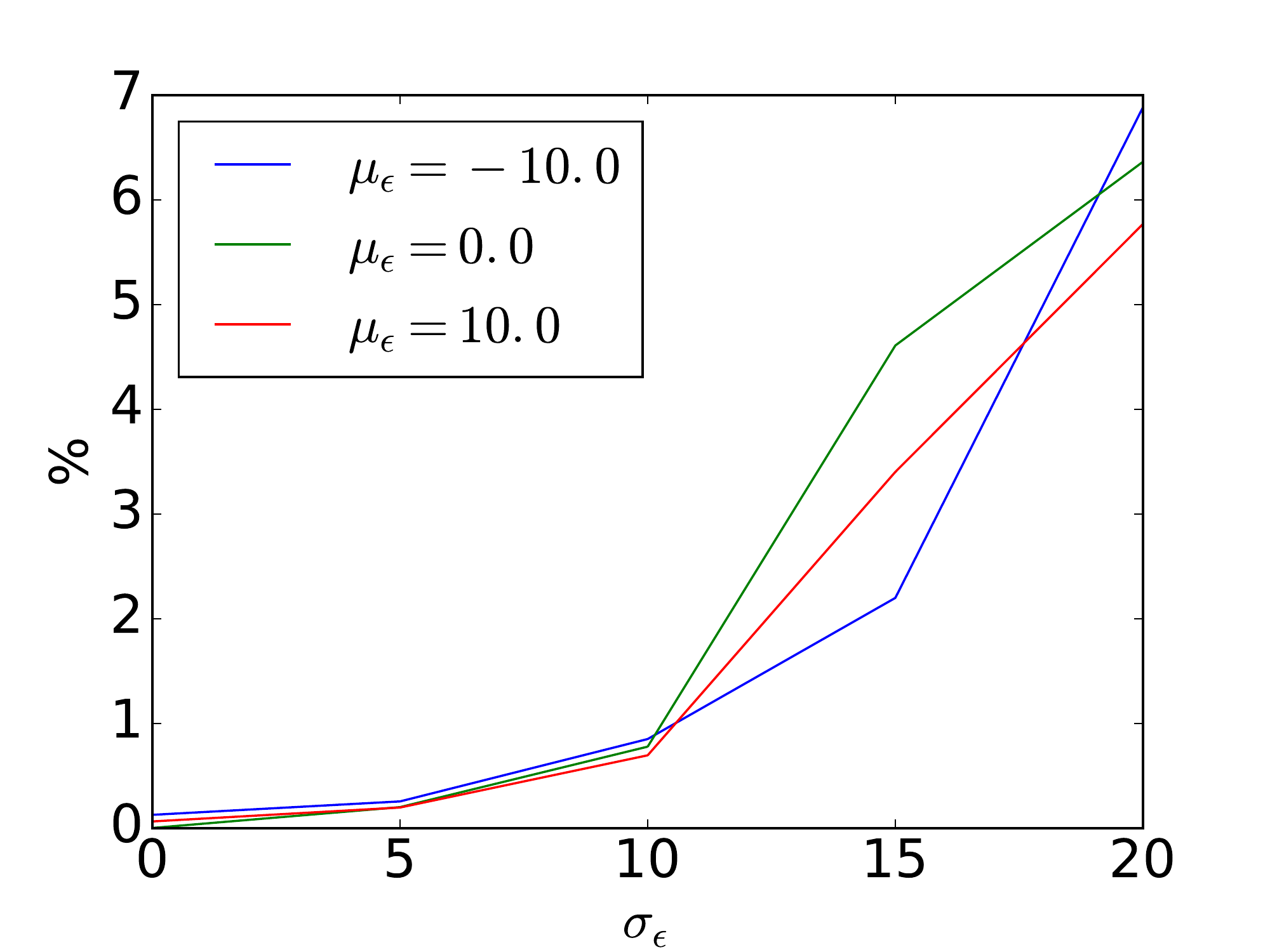}}
\caption{Percentage change in average total costs under $ED$-$FR$.}
\label{fig:inc-total-cost}
\end{minipage}
\end{figure*}

Observe the asymmetry in the plots between the different choices of $\mu_\epsilon$. In particular, when demand is constant ($\mu_d = 0$) or increasing ($\mu_d = 0.0002$), the sub-optimality is less sensitive to negative prediction errors ($\mu_\epsilon = -10$) than to positive prediction errors ($\mu_\epsilon = 10$). This phenomenon is due to high utilization of regulation resources at the optimal solution. Since majority of the regulation resources are hydro resources with low marginal costs, almost all the regulation resources are dispatched in the first time period, and there are less than 30MW of unused regulation capacity. Recall that negative prediction errors create incentive for economic dispatch to use less-than-optimal dispatch resources and more-than-optimal regulation resources. However, in the scenario with increasing demand, system demand increases by up to about 30MW. This deviation must be met by regulation resources, and since there are only 30MW of unused regulation capacity, the economic dispatch mechanism is unable to significantly increase the usage of regulation resources in the first time period, and therefore the sub-optimality is small under negative prediction errors. On the other hand, at the optimal solution, there is significant excess capacity of dispatch resources. Hence, positive prediction errors could lead to significantly more usage of dispatch resources and less usage of regulation resources, and cause a larger increase in total costs. We do not observe this asymmetry when demand is decreasing. This is due to the fact that $\delta_n^* = -80.34$ and hence a larger positive prediction error is needed for the asymmetry to manifest.

The simulations illustrate the complex interactions between $\boldsymbol\delta$ and the performance of $ED$-$FR$. In particular, both marginal costs and feasibility constraints have crucial impacts on the performance of $ED$-$FR$.

\section{conclusion}

This paper proposes an optimization decomposition approach for co-optimizing economic dispatch and frequency regulation resources. It demonstrates that optimization decomposition provides a rigorous way to design power system operations to allocate resources efficiently across timescales. Our main result, in Theorem~\ref{thm:decomposition}, shows one way to choose generation setpoints optimally at the economic dispatch timescale, and provides a guide on how to design a principled architecture for power system operations. In particular, using this result, we design an optimal frequency control scheme and an optimal economic dispatch mechanism, both of which differ from existing approaches in crucial ways and reveal potential inefficiencies in the latter. Hence, this paper underscores the need to jointly analyze economic dispatch and frequency regulation mechanisms when investigating the efficiency of the overall system.


\bibliographystyle{IEEEtran}
\bibliography{IEEEabrv,bibfile}

\begin{thebibliography}{10}
\providecommand{\url}[1]{#1}
\csname url@rmstyle\endcsname
\providecommand{\newblock}{\relax}
\providecommand{\bibinfo}[2]{#2}
\providecommand\BIBentrySTDinterwordspacing{\spaceskip=0pt\relax}
\providecommand\BIBentryALTinterwordstretchfactor{4}
\providecommand\BIBentryALTinterwordspacing{\spaceskip=\fontdimen2\font plus
\BIBentryALTinterwordstretchfactor\fontdimen3\font minus
  \fontdimen4\font\relax}
\providecommand\BIBforeignlanguage[2]{{%
\expandafter\ifx\csname l@#1\endcsname\relax
\typeout{** WARNING: IEEEtran.bst: No hyphenation pattern has been}%
\typeout{** loaded for the language `#1'. Using the pattern for}%
\typeout{** the default language instead.}%
\else
\language=\csname l@#1\endcsname
\fi
#2}}

\bibitem{CaiMalladaWierman2015a}
D.~Cai, E.~Mallada, and A.~Wierman, ``Distributed optimization decomposition
  for joint economic dispatch and frequency regulation,'' in \emph{2015 54th
  IEEE Conference on Decision and Control}, Dec. 2015, pp. 15--22.

\bibitem{WoodWollenberg1996}
A.~J. Wood and B.~F. Wollenberg, \emph{Power Generation, Operation, and
  Control}, 2nd~ed.\hskip 1em plus 0.5em minus 0.4em\relax John Wiley \& Sons,
  Inc., 1996.

\bibitem{BergenVittal2000}
A.~R. Bergen and V.~Vittal, \emph{Power Systems Analysis}, 2nd~ed.\hskip 1em
  plus 0.5em minus 0.4em\relax Prentice Hall, 2000.

\bibitem{MachowskiBialek2008}
J.~Machowski, J.~Bialek, and J.~Bumby, \emph{Power system dynamics: Stability
  and Control}, 2nd~ed.\hskip 1em plus 0.5em minus 0.4em\relax John Wiley \&
  Sons, Inc., 2008.

\bibitem{carpentier1979optimal}
J.~Carpentier, ``Optimal power flows,'' \emph{International Journal of
  Electrical Power \& Energy Systems}, vol.~1, no.~1, pp. 3--15, 1979.

\bibitem{kirschen2004front}
D.~Kirschen and G.~Strbac, \emph{Fundamentals of Power System Economics}.\hskip
  1em plus 0.5em minus 0.4em\relax Wiley Online Library, 2004.

\bibitem{baldick2004theory}
R.~Baldick, R.~Grant, and E.~Kahn, ``Theory and application of linear supply
  function equilibrium in electricity markets,'' \emph{Journal of Regulatory
  Economics}, vol.~25, no.~2, pp. 143--167, 2004.

\bibitem{caisoOptimization}
C.~Inc., ``{Market Optimization Details},''
  \url{http://caiso.com/Documents/TechnicalBulletin-MarketOptimizationDetails.pdf},
  November 2009, [Online; accessed Mar-24-2015].

\bibitem{schweppe1988spot}
F.~C. Schweppe, R.~D. Tabors, M.~Caraminis, and R.~E. Bohn, ``Spot pricing of
  electricity,'' 1988.

\bibitem{IbraheeKkumarKothari2005}
I.~Ibraheem, P.~Kumar, and D.~Kothari, ``Recent philosophies of automatic
  generation control strategies in power systems,'' \emph{Power Systems, IEEE
  Transactions on}, vol.~20, no.~1, pp. 346--357, Feb 2005.

\bibitem{deMello:1973jy}
F.~deMello, R.~Mills, and W.~B'Rells, ``{Automatic Generation Control Part
  II-Digital Control Techniques},'' \emph{Power Apparatus and Systems, IEEE
  Transactions on}, vol. PAS-92, no.~2, pp. 716--724, 1973.

\bibitem{ThatteZhangXie2011}
A.~A. Thatte, F.~Zhang, and L.~Xie, ``Frequency aware economic dispatch,'' in
  \emph{North American Power Symposium (NAPS), 2011}, Aug. 2011, pp. 1--7.

\bibitem{ZhaoMalladaLowBialek2016}
C.~Zhao, E.~Mallada, S.~Low, and J.~Bialek, ``A unified framework for frequency
  control and congestion management,'' in \emph{Power Systems Computation
  Conference, 2016}, to appear.

\bibitem{2014arXiv1410.2931M}
E.~{Mallada}, C.~{Zhao}, and S.~H. {Low}, ``{Optimal load-side control for
  frequency regulation in smart grids},'' \emph{ArXiv e-prints}, Oct. 2014.

\bibitem{LiChenZhao2014}
N.~Li, L.~Chen, and C.~Zhao, ``Connecting automatic generation control and
  economic dispatch from an optimization view,'' \emph{Control of Network
  Systems, IEEE Transactions on}, to appear.

\bibitem{CarpentierGohenCulioliRenaud1996}
P.~Carpentier, G.~Gohen, J.-C. Culioli, and A.~Renaud, ``Stochastic
  optimization of unit commitment: a new decomposition framework,'' \emph{IEEE
  Trans. on Power Systems}, vol.~11, no.~2, pp. 1067--1073, 1996.

\bibitem{TakritiBirgeLong1996}
S.~Takriti, J.~Birge, and E.~Long, ``A stochastic model for the unit commitment
  problem,'' \emph{IEEE Transactions on Power Systems}, vol.~11, no.~3, pp.
  1497--1508, 1996.

\bibitem{OzturkMazumdarNorman2004}
U.~Ozturk, M.~Mazumdar, and B.~Norman, ``A solution to the stochastic unit
  commitment problem using chance constrained programming,'' \emph{IEEE Trans.
  on Power Systems}, vol.~19, no.~3, pp. 1589--1598, 2004.

\bibitem{ElaOMalley2012}
E.~Ela and M.~O'Malley, ``Studying the variability and uncertainty impacts of
  variable generation at multiple timescales,'' \emph{IEEE Transactions on
  Power Systems}, vol.~27, no.~3, pp. 1324--1333, Aug. 2012.

\bibitem{Minoux2011}
M.~Minoux, ``On 2-stage robust lp with rhs uncertainty: complexity results and
  applications,'' \emph{Journal of Global Optimization}, vol.~49, no.~3, pp.
  521--537, 2011.

\bibitem{chiang2007layering}
M.~Chiang, S.~H. Low, A.~R. Calderbank, and J.~C. Doyle, ``Layering as
  optimization decomposition: A mathematical theory of network architectures,''
  \emph{Proc. of the IEEE}, vol.~95, no.~1, pp. 255--312, 2007.

\bibitem{palomar2006tutorial}
D.~P. Palomar and M.~Chiang, ``A tutorial on decomposition methods for network
  utility maximization,'' \emph{Selected Areas in Communications, IEEE Journal
  on}, vol.~24, no.~8, pp. 1439--1451, 2006.

\bibitem{cai2012optimal}
D.~W. Cai, C.~W. Tan, and S.~H. Low, ``Optimal max-min fairness rate control in
  wireless networks: Perron-frobenius characterization and algorithms,'' in
  \emph{INFOCOM, 2012 Proc. IEEE}, 2012, pp. 648--656.

\bibitem{BergenHill1981}
A.~R. Bergen and D.~J. Hill, ``A structure preserving model for power system
  stability analysis,'' \emph{{IEEE} Trans. Power App. Syst.}, no.~1, pp.
  25--35, 1981.

\bibitem{deMello:1973ea}
F.~deMello, R.~Mills, and W.~B'Rells, ``{Automatic Generation Control Part
  I-Process Modeling},'' \emph{Power Apparatus and Systems, IEEE Transactions
  on}, vol. PAS-92, no.~2, pp. 710--715, 1973.

\bibitem{ZhaoTopcuLiLow2014}
C.~Zhao, U.~Topcu, N.~Li, and S.~Low, ``{Design and Stability of Load-Side
  Primary Frequency Control in Power Systems},'' \emph{Automatic Control, IEEE
  Transactions on}, vol.~59, no.~5, pp. 1177--1189, 2014.

\bibitem{zhao2015acc}
C.~Zhao, E.~Mallada, and F.~Dorfler, ``Distributed frequency control for
  stability and economic dispatch in power networks,'' in \emph{American
  Control Conference (ACC), 2015}, July 2015, pp. 2359--2364.

\bibitem{YouChen2014ReverseAndForward}
S.~You and L.~Chen, ``Reverse and forward engineering of frequency control in
  power networks,'' in \emph{53th IEEE Conference on Decision and Control}, Dec
  2014.

\bibitem{Jokic2009}
A.~Jokic, M.~Lazar, and P.~van~den Bosch, ``On constrained steady-state
  regulation: Dynamic kkt controllers,'' \emph{Automatic Control, IEEE
  Transactions on}, vol.~54, no.~9, pp. 2250--2254, Sept 2009.

\bibitem{Cherukuri201610}
A.~Cherukuri, E.~Mallada, and J.~Cortes, ``{Asymptotic convergence of
  constrained primal{\textendash}dual dynamics},'' \emph{Systems {\&} Control
  Letters}, vol.~87, pp. 10--15, 2016.

\bibitem{rudkevich2005supply}
A.~Rudkevich, ``On the supply function equilibrium and its applications in
  electricity markets,'' \emph{Decision Support Systems}, vol.~40, no.~3, pp.
  409--425, 2005.

\bibitem{baldick2002electricity}
R.~Baldick, ``Electricity market equilibrium models: The effect of
  parametrization,'' \emph{IEEE Transactions on Power Systems}, vol.~17, no.~4,
  pp. 1170--1176, 2002.

\bibitem{johari2011parameterized}
R.~Johari and J.~N. Tsitsiklis, ``Parameterized supply function bidding:
  Equilibrium and efficiency,'' \emph{Operations research}, vol.~59, no.~5, pp.
  1079--1089, 2011.

\bibitem{klemperer1989supply}
P.~D. Klemperer and M.~A. Meyer, ``Supply function equilibria in oligopoly
  under uncertainty,'' \emph{Econometrica: Journal of the Econometric Society},
  pp. 1243--1277, 1989.

\bibitem{mas1995microeconomic}
A.~Mas-Colell, M.~D. Whinston, J.~R. Green, \emph{et~al.}, \emph{Microeconomic
  theory}.\hskip 1em plus 0.5em minus 0.4em\relax Oxford university press New
  York, 1995, vol.~1.

\bibitem{johari2009efficiency}
R.~Johari and J.~N. Tsitsiklis, ``Efficiency of scalar-parameterized
  mechanisms,'' \emph{Operations Research}, vol.~57, no.~4, pp. 823--839, 2009.

\bibitem{wang2012dynamic}
G.~Wang, M.~Negrete-Pincetic, A.~Kowli, E.~Shafieepoorfard, S.~Meyn, and U.~V.
  Shanbhag, ``Dynamic competitive equilibria in electricity markets,'' in
  \emph{Control and optimization methods for electric smart grids}.\hskip 1em
  plus 0.5em minus 0.4em\relax Springer, 2012, pp. 35--62.

\bibitem{WongAlbrechtAllanBillintonChenFongHaddadLiMukerjiPattonothers1999}
P.~Wong, P.~Albrecht, R.~Allan, R.~Billinton, Q.~Chen, C.~Fong, S.~Haddad,
  W.~Li, R.~Mukerji, D.~Patton, \emph{et~al.}, ``The ieee reliability test
  system-1996. a report prepared by the reliability test system task force of
  the application of probability methods subcommittee,'' \emph{Power Systems,
  IEEE Transactions on}, vol.~14, no.~3, pp. 1010--1020, 1999.

\end{thebibliography}

\section*{APPENDIX}

\begin{proof}[Proof of Theorem~\ref{thm:decomposition}]

The result follows from analyzing the Karush-Kuhn-Tucker (KKT) conditions of $SYSTEM$, $ED$, and $FR$. However, we first reformulate the problems as the notations are simpler with the reformulations. Define $\mathbf{q}_s^p := \mathbf{q}^p + \mathbf{r}_s^p$. Note that, due to the constraint that $\mathbf{r}_1^p = \mathbf{0}$, there is a bijection between the set of feasible $(\mathbf{q}^b,\mathbf{q}^p,\mathbf{r}^p)$ and the set of feasible $(\mathbf{q}^b,\mathbf{q}_1^p,\ldots,\mathbf{q}_S^p)$. Hence, $SYSTEM$ can be reformulated as:
\begin{equation}
\begin{array}{rl}
\underset{{\mathbf{q}}^b,{\mathbf{q}}_1^p,\ldots,{\mathbf{q}}_S^p}{\mathrm{min}} &\!\! \displaystyle\sum_{s\in S} p_s \sum_{n\in N} \left( c_n^b(q_{n}^b) + c_n^p(q_{s,n}^p) \right)
\\
\mathrm{s.t.}\;\;\; &\!\! (\mathbf{q}^b,\mathbf{q}_1^p,\mathbf{q}_s^p-\mathbf{q}_1^p) \in \Omega(\mathbf{d}_s), \quad \forall s\in S.
\end{array}
\label{eq:system}
\end{equation}
Also, $ED(\mathbf{d}_1)$ can be reformulated as:
\begin{align}
&\begin{array}{rl}
\underset{{\mathbf{q}}^b,{\mathbf{q}}_1^p}{\mathrm{min}}\! & \displaystyle\sum_{n\in N} \left( K c_n^b(q_{n}^b) + K c_n^p(q_{1,n}^p) - \delta_n q_n^b\right)
\\
\mathrm{s.t.} & (\mathbf{q}^b,\mathbf{q}_1^p,\mathbf{0}) \in \Omega(\mathbf{d}_1).
\end{array}
\label{eq:ed}
\end{align}
And, $FR(\mathbf{q}^b,\mathbf{q}^p,\mathbf{d}_s)$ can be reformulated as:
\begin{align}
&\begin{array}{rl}
\underset{{\mathbf{q}}^p_s}{\mathrm{min}} & \displaystyle\sum_{n\in N} c_n^p(q_{s,n}^p)
\\
\mathrm{s.t.} & (\mathbf{q}^{b},\mathbf{q}_{1}^{p},\mathbf{q}_{s}^{p}-\mathbf{q}_{1}^{p}) \in \Omega(\mathbf{d}_s).
\end{array}
\label{eq:fr}
\end{align}
Hence, $SYSTEM$ can be optimally decomposed into $ED$-$FR$ if $(\mathbf{q}^b,\mathbf{q}_1^p,\ldots,\mathbf{q}_S^p)$ is an optimal solution to~\eqref{eq:system} if and only if $(\mathbf{q}^b,\mathbf{q}_1^p)$ is an optimal solution to~\eqref{eq:ed} and $\mathbf{q}_s^p$ is an optimal solution to~\eqref{eq:fr} for all $s\in S$. 

Next, we prove (a). It is easy to see that~\eqref{eq:system} has compact sub-level sets. Moreover, its objective function is strictly convex. Hence,~\eqref{eq:system} has a unique optimal solution. By similar arguments, we conclude that~\eqref{eq:ed} has a unique optimal solution, and that~\eqref{eq:fr} has a unique optimal solution if the set $\left\{\mathbf{q}_s^p\in\mathbb{R}^N: (\mathbf{q}^b,\mathbf{q}_1^p,\mathbf{q}_s^p-\mathbf{q}_1^p)\in\Omega(\mathbf{d}_s)\right\}$ is non-empty. Hence, to prove (a), it suffices to show the forward implication, that is, if~\eqref{eq:optimal-delta} holds, then $(\mathbf{q}^{b},\mathbf{q}_1^{p},\ldots,\mathbf{q}_S^p)$ is an optimal solution to~\eqref{eq:system} implies that $(\mathbf{q}^{b},\mathbf{q}_1^{p})$ is an optimal solution to~\eqref{eq:ed} and $\mathbf{q}_s^{p}$ is an optimal solution to~\eqref{eq:fr} for all $s\in S$. The reverse implication follows from the existence and uniqueness of the optimal solutions. 

Let the Lagrangian of~\eqref{eq:system} be denoted by:
\begin{align*}
&L(\mathbf{q}^b,\mathbf{q}_1^p, \ldots,\mathbf{q}_S^p, \ubar{\boldsymbol\xi}, \bar{\boldsymbol\xi}, \ubar{\boldsymbol\nu}, \bar{\boldsymbol\nu}, \ubar{\boldsymbol\mu}, \bar{\boldsymbol\mu}, \boldsymbol\lambda)
\\
&:=
\sum_{s\in S} p_s \sum_{n\in N} \left( c_n^b(q_n^b) + c_n^p(q_{s,n}^p)\right) + L^b(\mathbf{q}^b,\ubar{\boldsymbol\xi},\bar{\boldsymbol\xi})
\\
& 
\ + \sum_{s\in S} p_s L^p(\mathbf{q}_s^p, \ubar{\boldsymbol\nu}_s,\bar{\boldsymbol\nu}_s)
+ \sum_{s\in S} p_s L^f(\mathbf{q}^b,\mathbf{q}_s^p,\ubar{\boldsymbol\mu}_s,\bar{\boldsymbol\mu}_s)
\\
& \ - \sum_{s\in S} p_s \lambda_s \mathbf{1}^\top \left(\mathbf{q}^b + \mathbf{q}_s^p - \mathbf{d}_s\right),
\end{align*}
where:
\begin{align*}
L^b(\mathbf{q}^b,\ubar{\boldsymbol\xi},\bar{\boldsymbol\xi})
&:=
\ubar{\boldsymbol\xi}^\top \left(\ubar{\mathbf{q}}^b - \mathbf{q}^b\right) 
+ \bar{\boldsymbol\xi}^\top \left(\mathbf{q}^b - \bar{\mathbf{q}}^b\right)
\\
L^p(\mathbf{q}_s^p,\ubar{\boldsymbol\nu}_s,\bar{\boldsymbol\nu}_s)
&:=
\ubar{\boldsymbol\nu}_s^\top \left( \ubar{\mathbf{q}}^p - \mathbf{q}_s^p\right) + \bar{\boldsymbol\nu}_s^\top\left(\mathbf{q}_s^p - \bar{\mathbf{q}}^p\right)
\\
L^f(\mathbf{q}^b,\mathbf{q}_s^p,\ubar{\boldsymbol\mu}_s,\bar{\boldsymbol\mu}_s)
&:=
\ubar{\boldsymbol\mu}_s^\top \left(-\mathbf{f} - \mathbf{H}\left(\mathbf{q}^b + \mathbf{q}_s^p  - \mathbf{d}_s\right)\right)
\\
& \ + \bar{\boldsymbol\mu}_s^\top \left( \mathbf{H}\left(\mathbf{q}^b + \mathbf{q}_s^p - \mathbf{d}_s\right) - \mathbf{f} \right).
\end{align*}
Note that we scaled the constraints by their probabilities, and $\ubar{\boldsymbol\xi}\in\mathbb{R}_+^N$, $\bar{\boldsymbol\xi}\in\mathbb{R}_+^N$, $\ubar{\boldsymbol\nu}=(\ubar{\boldsymbol\nu}_s, s\in S)\in\mathbb{R}_+^{NS}$, $\bar{\boldsymbol\nu}=(\bar{\boldsymbol\nu}_s, s\in S)\in\mathbb{R}_+^{NS}$, $\ubar{\boldsymbol\mu}=(\ubar{\boldsymbol\mu}_s, s\in S)\in\mathbb{R}_+^{LS}$, $\bar{\boldsymbol\mu}=(\bar{\boldsymbol\mu}_s, s\in S)\in\mathbb{R}_+^{LS}$, $\boldsymbol\lambda=(\lambda_s, s\in S)\in\mathbb{R}^{S}$ are appropriate Lagrange multipliers. 

Since~\eqref{eq:system} has a convex objective and linear constraints, from the KKT conditions, we infer that $(\mathbf{q}^{b},\mathbf{q}_1^{p},\ldots,\mathbf{q}_S^p)$ is an optimal solution to~\eqref{eq:system} if and only if $(\mathbf{q}^{b},\mathbf{q}_1^{p},\mathbf{q}_s^p-\mathbf{q}_1^p)\in \Omega(\mathbf{d}_s)$ for all $s\in S$ and there exists $\ubar{\boldsymbol\xi},\bar{\boldsymbol\xi}\in \mathbb{R}_+^N, \ubar{\boldsymbol\nu},\bar{\boldsymbol\nu} \in \mathbb{R}_+^{NS}, \ubar{\boldsymbol\mu},\bar{\boldsymbol\mu} \in \mathbb{R}_+^{LS}, \boldsymbol\lambda \in \mathbb{R}^S$ such that:
\begin{subequations}
\label{eq:system-kkt}
\begin{align}
\hspace{-10pt}
\left(Kc_n^{b\prime}(q_n^{b}), n \in N\right) + \bar{\boldsymbol\xi} - \ubar{\boldsymbol\xi} - \sum_{s\in S} p_{s} \boldsymbol\pi(\lambda_{s},\ubar{\boldsymbol\mu}_{s},\bar{\boldsymbol\mu}_{s}) &= 0;
\label{eq:system-kkt-a}
\\
L^b(\mathbf{q}^{b},\ubar{\boldsymbol\xi},\bar{\boldsymbol\xi}) &= 0;
\label{eq:system-kkt-b}
\\
\hspace{-14pt}\left(c_n^{p\prime}(q_{s,n}^p), n\in N\right) + \bar{\boldsymbol\nu}_s - \ubar{\boldsymbol\nu}_s - \boldsymbol\pi(\lambda_s,\ubar{\boldsymbol\mu}_s,\bar{\boldsymbol\mu}_s) &= 0;
\label{eq:system-kkt-c}
\\
L^p(\mathbf{q}_s^{p},\ubar{\boldsymbol\nu}_s,\bar{\boldsymbol\nu}_s) &= 0;
\label{eq:system-kkt-d}
\\
L^f(\mathbf{q}^{b}, \mathbf{q}_s^{p},\ubar{\boldsymbol\mu}_s,\bar{\boldsymbol\mu}_s) &= 0,
\label{eq:system-kkt-e}
\end{align}
\end{subequations}
for all $s\in S$.

Similarly, $(\mathbf{q}^b,\mathbf{q}_1^p)$ is an optimal solution to~\eqref{eq:ed} if and only if $(\mathbf{q}^b, \mathbf{q}_1^p,\mathbf{0}) \in \Omega(\mathbf{d}_1)$ and there exists $\ubar{\boldsymbol\xi}, \bar{\boldsymbol\xi} \in \mathbb{R}_+^N, \ubar{\boldsymbol\nu}_1, \bar{\boldsymbol\nu}_1 \in \mathbb{R}_+^{N}, \ubar{\boldsymbol\mu}_1, \bar{\boldsymbol\mu}_1 \in \mathbb{R}_+^{L}, \lambda_1 \in \mathbb{R}$ such that:
\begin{subequations}
\label{eq:ed-kkt}
\begin{align}
\hspace{-4pt}
\left(Kc_n^{b\prime}(q_n^b), n\in N\right) + \bar{\boldsymbol\xi} - \ubar{\boldsymbol\xi} - \boldsymbol\pi(\lambda_{1},\ubar{\boldsymbol\mu}_{1},\bar{\boldsymbol\mu}_{1}) - \boldsymbol\delta &= 0;
\label{eq:ed-kkt-a}
\\
L^b(\mathbf{q}^b, \ubar{\boldsymbol\xi}, \bar{\boldsymbol\xi}) &= 0;
\label{eq:ed-kkt-b}
\\
\hspace{-3pt}\left(Kc_n^{p\prime}(q_{1,n}^p), n \in N\right) + \bar{\boldsymbol\nu}_1 - \ubar{\boldsymbol\nu}_1 - \boldsymbol\pi(\lambda_1,\ubar{\boldsymbol\mu}_1,\bar{\boldsymbol\mu}_1) &= 0;
\label{eq:ed-kkt-c}
\\
L^p(\mathbf{q}_1^{p},\ubar{\boldsymbol\nu}_1,\bar{\boldsymbol\nu}_1) &= 0;
\label{eq:ed-kkt-d}
\\
L^f(\mathbf{q}^{b}, \mathbf{q}_1^{p}, \ubar{\boldsymbol\mu}_1, \bar{\boldsymbol\mu}_1) &= 0.
\label{eq:ed-kkt-e}
\end{align}
\end{subequations}

And $\mathbf{q}_s^p$ is an optimal solution to~\eqref{eq:fr} if and only if $(\mathbf{q}^{b},\mathbf{q}_1^p,\mathbf{q}_s^p-\mathbf{q}_1^p) \in \Omega(\mathbf{d}_s)$ and there exists $\ubar{\boldsymbol\nu}_s, \bar{\boldsymbol\nu}_s \in \mathbb{R}_+^{N}, \ubar{\boldsymbol\mu}_s, \bar{\boldsymbol\mu}_s \in \mathbb{R}_+^{L}, \lambda_s \in \mathbb{R}$ such that:
\begin{subequations}
\label{eq:fr-kkt}
\begin{align}
\left(c_n^{p\prime}(q_{s,n}^p), n\in N\right) + \bar{\boldsymbol\nu}_s - \ubar{\boldsymbol\nu}_s - \boldsymbol\pi(\lambda_s,\ubar{\boldsymbol\mu}_s,\bar{\boldsymbol\mu}_s) &= 0;
\label{eq:fr-kkt-a}
\\
L^p(\mathbf{q}_s^{p},\ubar{\boldsymbol\nu}_s,\bar{\boldsymbol\nu}_s) &= 0;
\label{eq:fr-kkt-b}
\\
L^f(\mathbf{q}^{b}, \mathbf{q}_s^{p}, \ubar{\boldsymbol\mu}_s,\bar{\boldsymbol\mu}_s) &= 0.
\label{eq:fr-kkt-c}
\end{align}
\end{subequations}

Suppose $(\mathbf{q}^b,\mathbf{q}_1^p,\ldots,\mathbf{q}_S^p)$ is an optimal solution to~\eqref{eq:system} with associated Lagrange multipliers $(\ubar{\boldsymbol\xi},\bar{\boldsymbol\xi}, \ubar{\boldsymbol\nu}, \bar{\boldsymbol\nu}, \ubar{\boldsymbol\mu}, \bar{\boldsymbol\mu}, \boldsymbol\lambda)$. Note that $(\mathbf{q}^b,\mathbf{q}_1^p,\mathbf{0}) \in\Omega(\mathbf{d}_1)$. From the fact that the variables $(\mathbf{q}^b,\ubar{\boldsymbol\xi},\bar{\boldsymbol\xi},\allowbreak\ubar{\boldsymbol\mu},\bar{\boldsymbol\mu},\lambda)$ satisfy~\eqref{eq:system-kkt-a} and~\eqref{eq:optimal-delta} and the fact that $\sum_{s\in S}p_s = K$, we infer that the variables $(\mathbf{q}^b,\ubar{\boldsymbol\xi},\bar{\boldsymbol\xi},\allowbreak K\ubar{\boldsymbol\mu}_1,K\bar{\boldsymbol\mu}_1,K\lambda_1)$ satisfy~\eqref{eq:ed-kkt-a}. From the fact that $(\mathbf{q}^b,\mathbf{q}_s^p,\bar{\boldsymbol\xi},\ubar{\boldsymbol\xi},\bar{\boldsymbol\nu}_s,\ubar{\boldsymbol\nu}_s,\bar{\boldsymbol\mu}_s,\ubar{\boldsymbol\mu}_s,\lambda_s)$ satisfy~\eqref{eq:system-kkt-b}~--~\eqref{eq:system-kkt-e}, we infer that the variables $(\mathbf{q}^b,\mathbf{q}_1^p,\bar{\boldsymbol\xi},\ubar{\boldsymbol\xi},K\bar{\boldsymbol\nu}_1,K\ubar{\boldsymbol\nu}_1,K\bar{\boldsymbol\mu}_1,K\ubar{\boldsymbol\mu}_1,K\lambda_1)$ satisfy~\eqref{eq:ed-kkt-b}~--~\eqref{eq:ed-kkt-e}. Hence, $(\mathbf{q}^b,\mathbf{q}_1^p)$ is an optimal solution to~\eqref{eq:ed}. Note also that $(\mathbf{q}^b,\mathbf{q}_1^p,\mathbf{q}_s^p-\mathbf{q}_1^p)\in\Omega(\mathbf{d}_s)$ for all $s\in S$. From the fact that the variables $(\mathbf{q}^b,\mathbf{q}_s^p,\ubar{\boldsymbol\nu}_s,\bar{\boldsymbol\nu}_s,\ubar{\boldsymbol\mu}_s,\bar{\boldsymbol\mu}_s,\lambda_s)$ satisfy~\eqref{eq:system-kkt-c}~--~\eqref{eq:system-kkt-e}, we infer that those variables satisfy~\eqref{eq:fr-kkt}. Hence, $\mathbf{q}_s^p$ is an optimal solution to~\eqref{eq:fr} for all $s\in S$. 

Next, we prove (b). Let $(\mathbf{q}^b,\mathbf{q}_1^p,\ldots,\mathbf{q}_S^p)$ be a solution to~\eqref{eq:system} such that $(\mathbf{q}^b,\mathbf{q}_1^p)$ is a solution to~\eqref{eq:ed}. If $\ubar{q}_n^b < q_n^{b} < \bar{q}_n^b$ and $\ubar{q}_n^p < q_{1,n}^{p} < \bar{q}_n^p$, then the complementary slackness conditions imply that $\ubar{\xi}_n = \bar{\xi}_n = 0$ and $\ubar{\nu}_{1,n} = \bar{\nu}_{1,n} = 0$. From the KKT conditions of~\eqref{eq:system}, which are given by~\eqref{eq:system-kkt}, we infer that:
\begin{align}
K c_n^{b\prime}(q_n^{b}) - \sum_{s\in S} p_s \pi_n(\lambda_s, \ubar{\boldsymbol\mu}_s, \bar{\boldsymbol\mu}_s) = 0;
\label{eq:system-kkt-short-a}
\\
c_n^{p\prime}(q_{1,n}^{p}) - \pi_n(\lambda_1, \ubar{\boldsymbol\mu}_1, \bar{\boldsymbol\mu}_1) = 0,
\label{eq:system-kkt-short-b}
\end{align}
where $(\ubar{\boldsymbol\mu}_s,\bar{\boldsymbol\mu}_s,\boldsymbol\lambda)$ are the associated Lagrange multipliers. From the KKT conditions of~\eqref{eq:ed}, which are given by~\eqref{eq:ed-kkt}, we infer that:
\begin{align}
\label{eq:ed-kkt-short-a}
K c_n^{b\prime}(q_n^{b}) - \pi_n(\lambda_1^\prime,\ubar{\boldsymbol\mu}_1^\prime, \bar{\boldsymbol\mu}_1^\prime) - \delta_n = 0;
\\
\label{eq:ed-kkt-short-b}
K c_n^{p\prime}(q_{1,n}^{p}) - \pi_n(\lambda_1^\prime, \ubar{\boldsymbol\mu}_1^\prime, \bar{\boldsymbol\mu}_1^\prime) = 0,
\end{align}
where $(\ubar{\boldsymbol\mu}_s^\prime,\bar{\boldsymbol\mu}_s^\prime,\boldsymbol\lambda^\prime)$ are the associated Lagrange multipliers. It follows that:
\begin{align*}
\delta_n
&=
\sum_{s\in S}p_s \pi_n(\lambda_s,\ubar{\boldsymbol\mu}_s,\bar{\boldsymbol\mu}_s) - \pi_n(\lambda_1^\prime,\ubar{\boldsymbol\mu}_1^\prime,\bar{\boldsymbol\mu}_1^\prime)
\\
&=
\sum_{s\in S}p_s \pi_n(\lambda_s,\ubar{\boldsymbol\mu}_s,\bar{\boldsymbol\mu}_s) - K\pi_n(\lambda_1,\ubar{\boldsymbol\mu}_1,\bar{\boldsymbol\mu}_1)
\\
&=
\sum_{s\in S}p_s \left(\pi_n(\lambda_s,\ubar{\boldsymbol\mu}_s,\bar{\boldsymbol\mu}_s) - \pi_n(\lambda_1,\ubar{\boldsymbol\mu}_1,\bar{\boldsymbol\mu}_1)\right).
\end{align*}
The first equality follows from comparing~\eqref{eq:system-kkt-short-a} and~\eqref{eq:ed-kkt-short-a}. The second equality follows from comparing~\eqref{eq:system-kkt-short-b} and~\eqref{eq:ed-kkt-short-b}. The last equality follows from the fact that $\sum_{s\in S}p_s = K$.  
\end{proof}

\begin{proof}[Proof of Proposition~\ref{prop:optimality}]

We provide a proof sketch of this result. The skipped details can be found in~\cite{2014arXiv1410.2931M}.
(i) follows from the KKT conditions of $FR'(\mathbf{q}^b,\mathbf{q}^p,\mathbf{d}_s)$ and is shown in~\cite[Lemma 2]{2014arXiv1410.2931M}. Since $\boldsymbol\omega_s^{\prime}=0$, it follows from constraints~\eqref{eq:FR'_k-b} and~\eqref{eq:FR'_k-c} of $FR'(\mathbf{q}^b,\mathbf{q}^p,\mathbf{d}_s)$ that
$
\mathbf{L}\boldsymbol\theta_s^{\prime} = \mathbf{L}\boldsymbol\phi_s^{\prime},
$
which, since the null space of $\mathbf{L}$ is $\spn\{\mathbf{1}\}$, implies that $\boldsymbol\theta_s^{\prime}=\boldsymbol\phi_s^{\prime} + \alpha{\mathbf{1}}$ for some $\alpha \in \mathbb{R}$. This implies that $\mathbf{B}\mathbf{C}^\top\boldsymbol\phi_s^{\prime} = \mathbf{B}\mathbf{C}^\top\boldsymbol\theta_s^{\prime}$.  Therefore, without loss of generality, we can substitute constraint~\eqref{eq:FR'_k-b} in $FR'(\mathbf{q}^b,\mathbf{q}^p,\mathbf{d}_s)$ by the constraint $\boldsymbol\omega_s=0$. 
Then, using the definition of $\mathbf{H}$ and the equivalence between~\eqref{eq:supply-demand-constraint} and~\eqref{eq:supply-demand-constraint-2}, we infer that the feasible sets of $FR(\mathbf{q}^b,\mathbf{q}^p,\mathbf{d}_s)$ and $FR'(\mathbf{q}^b,\mathbf{q}^p,\mathbf{d}_s)$ are equivalent. Finally, since $c_n^p(\cdot)$ is strictly convex, by uniqueness of the optimal solutions, we get (ii). Lastly, (iii) follows from the definition of $\mathbf{H}$ and $\mathbf{B}\mathbf{C}^\top\boldsymbol\phi_s^{\prime}=\mathbf{B}\mathbf{C}^\top\boldsymbol\theta_s^{\prime}$. The final statement of the proposition follows directly from~\cite[Theorem 8]{2014arXiv1410.2931M}.
\end{proof}

\begin{proof}[Proof of Proposition~\ref{prop:market}]

Our proof proceeds in 6 steps: (1) Characterizing regulation generators' optimal bids $\boldsymbol\alpha^p$ given their prices $\boldsymbol\pi^p$; (2) Characterizing dispatch generators' optimal bids $\boldsymbol\alpha^b$ given their prices $\boldsymbol\pi^b$; (3) Characterizing prices $(\boldsymbol\pi^b,\boldsymbol\pi^p)$ given bids $(\boldsymbol\alpha^b,\boldsymbol\alpha^p)$ using KKT conditions; (4) Showing that, at an equilibrium, the production schedule is the unique optimal solution to $\hat{ED}$-$\hat{FR}$; (5) Showing that any production schedule $(\mathbf{q}^{b},\mathbf{q}^{p},\mathbf{r}^p)$ that solves $SYSTEM$ can be obtained using bids $(\boldsymbol\gamma^b,\boldsymbol\gamma^p)$ and the latter satisfy the equilibrium characterizations in steps 1 to 3; and (6) Showing that any bids $(\boldsymbol\alpha^b,\boldsymbol\alpha^p)$ that satisfy the equilibrium characterizations in steps 1 to 3 give the same production schedule as that under bids $(\boldsymbol\gamma^b,\boldsymbol\gamma^p)$ (which also solves $SYSTEM$). Note that part (a) follows from step 6 and part (b) follows from step 5.

\emph{Step 1: Characterizing regulation generators' optimal bids $\boldsymbol\alpha^p$ given their prices $\boldsymbol\pi^p$.} Since $c_n^p$ is strictly convex and $c_n^p(q_{s,n}^p)\!\rightarrow\!+\infty$ as $q_{s,n}^p\!\rightarrow\!\{\ubar{q}_n^p,\bar{q}_n^p\}$, $c_n^{p\prime}$ is invertible. Let $\sigma\!:\! S\!\rightarrow \!S$ be any permutation function that satisfies:
\begin{align*}
c_n^{p\prime-1}(\pi_{\sigma(1),n}^p) \leq c_n^{p\prime-1}(\pi_{\sigma(2),n}^p) \leq \ldots \leq c_n^{p\prime-1}(\pi_{\sigma(S),n}^p),
\end{align*}
and let integers $i,j\in \{0,1,\ldots,S\} $ be such that:
\begin{subequations}
\begin{align}
& \ c_n^{p\prime-1}(\pi_{\sigma(s),n}^p) \leq \ubar{q}_n^p && \forall s = 1,\ldots,i;
\label{eq:peaker-tight-lower-capacity}
\\
\ubar{q}_n^p < & \ c_n^{p\prime-1}(\pi_{\sigma(s),n}^p) < \bar{q}_n^p && \forall s = i+1,\ldots,j;
\label{eq:peaker-non-binding-capacity}
\\
\bar{q}_n^p \leq & \ c_n^{p\prime-1}(\pi_{\sigma(s),n}^p)  && \forall s = j+1,\ldots,S. 
\label{eq:peaker-tight-upper-capacity}
\end{align}
\end{subequations}
We now show that $\alpha_n^p \in\mathbb{R}_{++}$ maximizes $\mathrm{PF}_n^p$ if and only if: 
\begin{subequations}
\label{eq:peaker-optimal-bid}
\begin{align}
\hspace{-9pt}
\alpha_n^p s_n^{p}(\pi_{\sigma(s),n}^p) &\leq \ubar{q}_n^p && \forall s = 1,\ldots, i;
\label{eq:peaker-optimal-bid-a}
\\
\hspace{-9pt}
\alpha_n^p s_n^{p}(\pi_{\sigma(s),n}^p) &= c_n^{p\prime -1}(\pi_{\sigma(k),n}^p) && \forall s=i+1,\ldots,j;
\label{eq:peaker-optimal-bid-b}
\\
\hspace{-9pt}
\alpha_n^p s_n^{p}(\pi_{\sigma(s),n}^p) & \geq \bar{q}_n^p && \forall s = j+1,\ldots,S.
\label{eq:peaker-optimal-bid-c}
\end{align}
\end{subequations}

For notational brevity, in the rest of this step, we abuse notation and let: 
$$q_{s,n}^p(\alpha_n^p) =
[\alpha_n^p s_n^{p}(\pi_{\sigma(s),n}^p)]_{\ubar{q}_n^p}^{\bar{q}_n^p}.$$ 
To prove our characterization, it suffices to show that, given any $\alpha_n^p \in \mathbb{R}_{++}$ that satisfies~\eqref{eq:peaker-optimal-bid}, the vector of per-outcome profits
\begin{align}
& \left(\pi_{\sigma(s),n}^pq_{s,n}^p(\alpha_n^p) - c_n^p\left(q_{s,n}^p(\alpha_n^p)\right), s\in S\right)
\notag
\\
& \gneq
\left(\pi_{\sigma(s),n}^p q_{s,n}^p(\bar{\alpha}_n^p) - c_n^p\left(q_{s,n}^p(\bar{\alpha}_n^p)\right), s\in S\right)
\label{eq:peaker-bid-vector-inequality}
\end{align}
for any $\bar{\alpha}_n^p$ that does not satisfy~\eqref{eq:peaker-optimal-bid}. Since $p_{\sigma(s)} > 0$ for all $s\in S$, it then follows that:
\begin{align*}
\left. \mathrm{PF}_n^p \right|_{\alpha_n^p} 
&=
\sum_s p_{\sigma(s)} \left( \pi_{\sigma(s),n}^p q_{s,n}^p(\alpha_n^p) - c_n^p\left(q_{s,n}^p(\alpha_n^p)\right)\right)
\\
&>
\sum_s p_{\sigma(s)} \left(\pi_{\sigma(s),n}^p q_{s,n}^p(\bar{\alpha}_n^p) - c_n^p\left(q_{s,n}^p(\bar{\alpha}_n^p)\right)\right)
\\
&=
\left. \mathrm{PF}_n^p \right|_{\bar{\alpha}_n^p}.
\end{align*}

Suppose $s \in \{1,\ldots,i\}$. From~\eqref{eq:peaker-tight-lower-capacity} and the fact that $c_n^p$ is strictly convex, we infer that $\pi_{\sigma(s),n}^p \leq c_n^{p\prime}(\ubar{q}_n^p)$. From~\eqref{eq:peaker-optimal-bid-a}, we infer that $q_{s,n}^p(\alpha_n^p) = \ubar{q}_n^p$. Then:
\begin{align*}
&
c_n^p(q_{s,n}^p(\bar{\alpha}_n^p))
\\
&\geq 
c_n^p(\ubar{q}_n^p) + c_n^{p\prime}(\ubar{q}_n^p)\left(q_{s,n}^p(\bar{\alpha}_n^p) - \ubar{q}_n^p\right)
\\
&\geq 
c_n^p(\ubar{q}_n^p) + \pi_{\sigma(s),n}^p\left(q_{s,n}^p(\bar{\alpha}_n^p) - \ubar{q}_n^p\right)
\\
&=
c_n^p(q_{s,n}^p(\alpha_n^p)) + \pi_{\sigma(s),n}^p\left(q_{s,n}^p(\bar{\alpha}_n^p) - q_{s,n}^p(\alpha_n^p)\right),
\end{align*}
where the first inequality follows from the fact that $c_n^p$ is strictly convex, the second inequality follows from $\pi_{\sigma(s),n}^p \leq c_n^{p\prime}(\ubar{q}_n^p)$ and $q_{s,n}^p(\bar{\alpha}_n^p) \geq \ubar{q}_n^p$, and the last equality follows from $q_{s,n}^p(\alpha_n^p) = \ubar{q}_n^p$. Furthermore, if $q_{s,n}^p(\bar{\alpha}_n^p) > \ubar{q}_n^p$, then the first inequality is strict, and hence:
\begin{align*}
&c_n^p(q_{s,n}^p(\bar{\alpha}_n^p))
\\
&>
c_n^p(q_{s,n}^p(\alpha_n^p)) + \pi_{\sigma(s),n}^p\left(q_{s,n}^p(\bar{\alpha}_n^p) - q_{s,n}^p(\alpha_n^p)\right).
\end{align*}

Suppose $s \in \{i+1,\ldots,j\}$. From~\eqref{eq:peaker-non-binding-capacity} and~\eqref{eq:peaker-optimal-bid-b}, we infer that $q_{s,n}^p(\alpha_n^p) = c_n^{p\prime -1}(\pi_{\sigma(s),n}^p)$ and $\ubar{q}_n^p < q_{s,n}^p(\alpha_n^p) < \bar{q}_n^p$. From $\ubar{q}_n^p < q_{s,n}^p(\alpha_n^p) < \bar{q}_n^p$, and the fact that $s_n^{p}(\pi_{\sigma(s),n}^p) \neq 0$ and $\bar{\alpha}_n^p \neq \alpha_n^p$, we infer that $q_{s,n}^p(\bar{\alpha}_n^p) \neq q_{s,n}^p(\alpha_n^p)$. Then:
\begin{align*}
&c_n^p(q_{s,n}^p(\bar{\alpha}_n^p))
\\
&>
c_n^p(q_{s,n}^p(\alpha_n^p)) + c_n^{p\prime}(q_{s,n}^p(\alpha_n^p))\left(q_{s,n}^p(\bar{\alpha}_n^p)-q_{s,n}^p(\alpha_n^p)\right)
\\
&=
c_n^p(q_{s,n}^p(\alpha_n^p)) + \pi_{\sigma(s),n}^p\left(q_{s,n}^p(\bar{\alpha}_n^p) - q_{s,n}^p(\alpha_n^p)\right),
\end{align*}
where the first inequality follows from the fact that $c_n^p$ is strictly convex and $q_{s,n}^p(\bar{\alpha}_n^p) \neq q_{s,n}^p(\alpha_n^p)$ and the equality follows from $q_{s,n}^p(\alpha_n^p) = c_n^{p\prime -1}(\pi_{\sigma(s),n}^p)$.

Suppose $s \in \{i+1,\ldots,S\}$. From~\eqref{eq:peaker-tight-upper-capacity} and the fact that $c_n^p$ is strictly convex, we infer that $\pi_{\sigma(s),n}^p \geq c_n^{p\prime}(\bar{q}_n^p)$. From~\eqref{eq:peaker-optimal-bid-c}, we infer that $q_{s,n}^p(\alpha_n^p) = \bar{q}_n^p$. Then:
\begin{align*}
&
c_n^p(q_{s,n}^p(\bar{\alpha}_n^p))
\\
&\geq 
c_n^p(\bar{q}_n^p) + c_n^{p\prime}(\bar{q}_n^p)\left(q_{s,n}^p(\bar{\alpha}_n^p) - \bar{q}_n^p\right)
\\
&\geq 
c_n^p(\bar{q}_n^p) + \pi_{\sigma(s),n}^p\left(q_{s,n}^p(\bar{\alpha}_n^p) - \bar{q}_n^p\right)
\\
&=
c_n^p(q_{s,n}^p(\alpha_n^p)) + \pi_{\sigma(s),n}^p\left(q_{s,n}^p(\bar{\alpha}_n^p) - q_{s,n}^p(\alpha_n^p)\right),
\end{align*}
where the first inequality follows from the fact that $c_n^p$ is strictly convex, the second inequality follows from $\pi_{\sigma(s),n}^p \geq c_n^{p\prime}(\bar{q}_n^p)$ and $q_{s,n}^p(\bar{\alpha}_n^p) \leq \bar{q}_n^p$, and the last equality follows from $q_{s,n}^p(\alpha_n^p) = \bar{q}_n^p$. Furthermore, if $q_{s,n}^p(\bar{\alpha}_n^p) < \bar{q}_n^p$, then the first inequality is strict, and hence:
\begin{align*}
&
c_n^p(q_{s,n}^p(\bar{\alpha}_n^p))
\\
&>
c_n^p(q_{s,n}^p(\alpha_n^p)) + \pi_{\sigma(s),n}^p\left(q_{s,n}^p(\bar{\alpha}_n^p) - q_{s,n}^p(\alpha_n^p)\right).
\end{align*}

Hence, for all $s\in S$:
\begin{align}
&c_n^p(q_{s,n}^p(\bar{\alpha}_n^p))
\notag
\\
&\geq
c_n^p(q_{s,n}^p(\alpha_n^p)) + \pi_{\sigma(s),n}^p\left(q_{s,n}^p(\bar{\alpha}_n^p - q_{s,n}^p(\alpha_n^p)\right).
\label{eq:peaker-optimal-bid-inequality}
\end{align}
Moreover, this inequality is strict for some $s\in S$. If $i < j$, the inequality is strict for $s \in \{i+1,\ldots,j\}$. If $i = j$, then, since $\bar{\alpha}_n^p$ does not satisfy~\eqref{eq:peaker-optimal-bid}, there exists some $s\in\{1,\ldots,i\}$ such that $\alpha_n^p s_n^p(\pi_{\sigma(s),n}^p) > \ubar{q}_n^p$ or some $s\in\{i+1,\ldots,S\}$ such that $\alpha_n^p s_n^p(\pi_{\sigma(s),n}^p) < \bar{q}_n^p$,  
and hence there exists some $s\in\{1,\ldots,i\}$ such that $q_{s,n}^p(\bar{\alpha}_n^p) > \ubar{q}_n^p$ or some $s\in\{i+1,\ldots,S\}$ such that $q_{s,n}^p(\bar{\alpha}_n^p) < \bar{q}_n^p$, and the inequality in~\eqref{eq:peaker-optimal-bid-inequality} is strict for that $s$. Hence, we conclude that:
\begin{align*}
& \left(c_n^p(q_{s,n}^p(\bar{\alpha}_n^p)), s\in S \right)
\\
& \gneq
\left(c_n^p(q_{s,n}^p(\alpha_n^p)) + \pi_{\sigma(s),n}^p\left(q_{s,n}^p(\bar{\alpha}_n^p) - q_{s,n}^p(\alpha_n^p)\right), s\in S\right)
\end{align*}
for any $\bar{\alpha}_n^p$ that does not satisfy~\eqref{eq:peaker-optimal-bid}. By rearranging terms, we obtain~\eqref{eq:peaker-bid-vector-inequality}.

\emph{Step 2: Characterizing dispatch generators' optimal bids $\boldsymbol\alpha^b$ given their prices $\boldsymbol\pi^b$.} Note that the profit maximization problem for a dispatch generator is a special case of that for a regulation generator with $S = 1$. By applying the characterization in step 1, we infer that $\alpha_n^b\in \mathbb{R}_{++}$ maximizes $\mathrm{PF}_n^b$ if and only if:
\begin{subequations}
\label{eq:baseload-optimal-bid}
\begin{align}
\alpha_n^b s_n^{b}(\pi_{n}^b) & \leq \ubar{q}_n^b, && \text{if} \quad\quad\,\,\,\, c_n^{b\prime-1}(\pi_{n}^b) \leq \ubar{q}_n^b;
\label{eq:baseload-optimal-bid-a}
\\
\alpha_n^b &= \gamma_n^b, &&  \text{if} \ \ubar{q}_n^b < c_n^{b\prime-1}(\pi_{n}^b) < \bar{q}_n^b;
\label{eq:baseload-optimal-bid-b}
\\
\alpha_n^b s_n^{b}(\pi_{n}^b) & \geq \bar{q}_n^b, &&  \text{if} \ \bar{q}_n^b \leq c_n^{b\prime-1}(\pi_{n}^b).
\label{eq:baseload-optimal-bid-c}
\end{align}
\end{subequations}

\emph{Step 3: Characterizing prices $(\boldsymbol\pi^b,\boldsymbol\pi^p)$ given bids $(\boldsymbol\alpha^b,\boldsymbol\alpha^p)$ using KKT conditions.} First, we take the same approach as in the proof of Theorem~\ref{thm:decomposition} and reformulate $\hat{ED}$ and $\hat{FR}$ before applying the KKT conditions. Relabeling the variable $\mathbf{q}^p$ to $\mathbf{q}_1^p$ in $\hat{ED}$ gives:
\begin{align}
&\begin{array}{rl}
\underset{{\mathbf{q}}^b,{\mathbf{q}}_1^p}{\mathrm{min}}\! & \displaystyle\sum_{n\in N} \left( K \hat{c}_n^b(q_{n}^b) + K \hat{c}_n^p(q_{1,n}^p) - \delta_n q_n^b\right)
\\
\mathrm{s.t.} & (\mathbf{q}^b,\mathbf{q}_1^p,\mathbf{0}) \in \Omega(\mathbf{d}_1).
\end{array}
\label{eq:edhat}
\end{align}
And substituting $\mathbf{q}_s^p = \mathbf{q}^p + \mathbf{r}_s^p$ in $\hat{FR}$ gives:
\begin{align}
&\begin{array}{rl}
\underset{{\mathbf{q}}^p_s}{\mathrm{min}} & \displaystyle\sum_{n\in N} \hat{c}_n^p(q_{s,n}^p)
\\
\mathrm{s.t.} & (\mathbf{q}^{b},\mathbf{q}_{1}^{p},\mathbf{q}_{s}^{p}-\mathbf{q}_{1}^{p}) \in \Omega(\mathbf{d}_s).
\end{array}
\label{eq:frhat}
\end{align}

Substituting $s_n^b = c_n^{b\prime-1}(\cdot)/\gamma_n^b$ and $s_n^p = c_n^{p\prime-1}(\cdot)/\gamma_n^p$ into the definition of $\hat{c}_n^b$ and $\hat{c}_n^p$ implies that:
\begin{align*}
\hat{c}_n^b(q_n^b)
&=
\int_{\ubar{q}_n^b}^{q_n^b} c_n^{b\prime}((\gamma_n^b/\alpha_n^b)w)\;dw,
\\
\hat{c}_n^p(q_n^p)
&=
\int_{\ubar{q}_n^p}^{q_n^p} c_n^{p\prime}((\gamma_n^p/\alpha_n^p)w)\;dw.
\end{align*}
Hence,~\eqref{eq:edhat} has a continuous and strictly convex objective and linear constraints. Thus, from the KKT conditions, $(\mathbf{q}^b,\mathbf{q}_1^p)$ is an optimal solution to~\eqref{eq:edhat} if and only if $(\mathbf{q}^b, \mathbf{q}_1^p,\mathbf{0}) \in \Omega(\mathbf{d}_1)$ and there exists $\ubar{\boldsymbol\xi}, \bar{\boldsymbol\xi} \in \mathbb{R}_+^N, \ubar{\boldsymbol\nu}_1, \bar{\boldsymbol\nu}_1 \in \mathbb{R}_+^{N}, \ubar{\boldsymbol\mu}_1, \bar{\boldsymbol\mu}_1 \in \mathbb{R}_+^{L}, \lambda_1 \in \mathbb{R}$ such that:
\begin{subequations}
\label{eq:market-ed-kkt}
\begin{align}
\left(K c_n^{b\prime}((\gamma_n^b/\alpha_n^b)q_n^b), n\in N\right) + \bar{\boldsymbol\xi} - \ubar{\boldsymbol\xi} - K\boldsymbol\pi^b &= 0;
\label{eq:market-ed-baseload-foc}
\\
L^b(\mathbf{q}^b, \ubar{\boldsymbol\xi}, \bar{\boldsymbol\xi}) &= 0;
\label{eq:market-ed-baseload-comp-slack}
\\
\hspace{-2pt}\left(K c_n^{p\prime}((\gamma_n^p/\alpha_n^p)q_{1,n}^p), n\in N\right) + \bar{\boldsymbol\nu}_1 - \ubar{\boldsymbol\nu}_1 - K\boldsymbol\pi_1^p &= 0;
\label{eq:market-ed-peaker-foc}
\\
L^p(\mathbf{q}_1^{p},\ubar{\boldsymbol\nu}_1,\bar{\boldsymbol\nu}_1) &= 0;
\label{eq:market-ed-peaker-comp-slack}
\\
L^f(\mathbf{q}^{b}, \mathbf{q}_1^{p}, \ubar{\boldsymbol\mu}_1,\bar{\boldsymbol\mu}_1) &= 0,
\end{align}
where:
\begin{align}
\boldsymbol\pi^b 
&= (1/K)\left(\boldsymbol\pi(\lambda_1,\ubar{\boldsymbol\mu}_1,\bar{\boldsymbol\mu}_1) + \boldsymbol\delta\right);
\label{eq:market-ed-baseload-price}
\\
\boldsymbol\pi_1^p
&= (1/K)\boldsymbol\pi(\lambda_1,\ubar{\boldsymbol\mu}_1,\bar{\boldsymbol\mu}_1).
\label{eq:market-ed-peaker-price}
\end{align}
\end{subequations}
Similarly, from the KKT conditions, $\mathbf{q}_s^p$ is an optimal solution to~\eqref{eq:frhat} if and only if $(\mathbf{q}^{b},\mathbf{q}_1^p,\mathbf{q}_s^p-\mathbf{q}_1^p) \in \Omega(\mathbf{d}_s)$ and there exists $\ubar{\boldsymbol\nu}_s, \bar{\boldsymbol\nu}_s \in \mathbb{R}_+^{N}, \ubar{\boldsymbol\mu}_s, \bar{\boldsymbol\mu}_s \in \mathbb{R}_+^{L}, \lambda_s \in \mathbb{R}$ such that:
\begin{subequations}
\label{eq:market-fr-kkt}
\begin{align}
\left(c_n^{p\prime}((\gamma_n^p/\alpha_n^p)q_{s,n}^p), n\in N\right) + \bar{\boldsymbol\nu}_s - \ubar{\boldsymbol\nu}_s - \boldsymbol\pi_s^p &= 0;
\label{eq:market-fr-peaker-foc}
\\
L^p(\mathbf{q}_s^{p},\ubar{\boldsymbol\nu}_s,\bar{\boldsymbol\nu}_s) &= 0;
\\
L^f(\mathbf{q}^{b}, \mathbf{q}_s^{p}, \ubar{\boldsymbol\mu}_s,\bar{\boldsymbol\mu}_s) &= 0,
\end{align}
where:
\begin{align}
\boldsymbol\pi_s^p = \boldsymbol\pi(\lambda_s,\ubar{\boldsymbol\mu}_s,\boldsymbol\mu_s).
\label{eq:market-fr-peaker-price}
\end{align}
\end{subequations}

\emph{Step 4: Showing that, at an equilibrium, the production schedule is the unique optimal solution to $\hat{ED}$-$\hat{FR}$.} Let $(\mathbf{q}^b,\mathbf{q}^p)$ be an optimal solution to $\hat{ED}(\mathbf{d}_1)$ and $\mathbf{r}_s^p$ be an optimal solution to $\hat{FR}(\mathbf{q}^b,\mathbf{q}^p,\mathbf{d}_s)$. We will show that:
\begin{align*}
\mathbf{q}^b
&=
\left([\alpha_n^b s_n^{b}(\pi_{n}^b)]_{\ubar{q}_n^b}^{\bar{q}_n^b}, n\in N\right);
\\
\mathbf{q}^p
&=
\left([\alpha_n^p s_n^{p}(\pi_{1,n}^p)]_{\ubar{q}_n^p}^{\bar{q}_n^p}, n\in N\right);
\\
\mathbf{r}_{s}^p
&=
\left([\alpha_n^p s_n^{p}(\pi_{s,n}^p)]_{\ubar{q}_n^p}^{\bar{q}_n^p} - [\alpha_n^p s_n^{p}(\pi_{1,n}^p)]_{\ubar{q}_n^p}^{\bar{q}_n^p}, n\in N\right).
\end{align*}
It suffices to show that, if $(\mathbf{q}^b,\mathbf{q}_1^p)$ is an optimal solution to~\eqref{eq:edhat} and $\mathbf{q}_s^p$ is an optimal solution to~\eqref{eq:frhat}, then:
\begin{align}
\mathbf{q}^b
&=
\left([\alpha_n^b s_n^{b}(\pi_{n}^b)]_{\ubar{q}_n^b}^{\bar{q}_n^b}, n\in N\right);
\label{eq:baseload-dispatch-optimal-solution}
\\
\mathbf{q}_{s}^p
&=
\left([\alpha_n^p s_n^{p}(\pi_{s,n}^p)]_{\ubar{q}_n^p}^{\bar{q}_n^p}, n\in N\right).
\label{eq:peaker-dispatch-optimal-solution}
\end{align}
By rewriting~\eqref{eq:market-ed-baseload-foc} for dispatch generator $n$, we infer that:
\begin{align*}
q_n^b
=
\alpha_n^b s_n^{b}\left(\pi_n^b + \ubar{\xi}_n/K - \bar{\xi}_n/K\right).
\end{align*}
If $\ubar{q}_n^b < q_n^b < \bar{q}_n^b$, then from~\eqref{eq:market-ed-baseload-comp-slack}, we infer that $\bar{\xi}_n = \ubar{\xi}_n = 0$, which implies that $q_n^b = \alpha_n^b s_n^{b}(\pi_n^b)$. If $q_n^b = \ubar{q}_n^b$, then from~\eqref{eq:market-ed-baseload-comp-slack}, we infer that $\bar{\xi}_n = 0$ and $\ubar{\xi}_n \geq 0$, which implies that $\ubar{q}_n^b = q_n^b = \alpha_n^b s_n^{b}(\pi_n^b + \ubar{\xi}_n/K)\geq \alpha_n^b s_n^{b}(\pi_n^b)$, where the last inequality follows from the fact that $c_n^b$ is strictly convex. If $q_n^b = \bar{q}_n^b$, then from~\eqref{eq:market-ed-baseload-comp-slack}, we infer that $\ubar{\xi}_n = 0$ and $\bar{\xi}_n \geq 0$, which implies that $\bar{q}_n^b = q_n^b = \alpha_n^b s_n^{b}(\pi_n^b - \bar{\xi}_n/K) \leq \alpha_n^b s_n^{b}(\pi_n^b)$, where the last inequality follows from the fact that $c_n^b$ is strictly convex. Hence, we conclude that $\mathbf{q}^b$ is given by~\eqref{eq:baseload-dispatch-optimal-solution}. By making similar arguments, we conclude that $\mathbf{q}_{s}^p$ is given by~\eqref{eq:peaker-dispatch-optimal-solution}.

\emph{Step 5: Showing that any production schedule $(\mathbf{q}^{b},\mathbf{q}^{p},\mathbf{r}^p)$ that solves $SYSTEM$ can be obtained using bids $(\boldsymbol\gamma^b,\boldsymbol\gamma^p)$ and the latter satisfy the characterizations in steps 1 to 3.} By Theorem~\ref{thm:decomposition}, $(\mathbf{q}^{b},\mathbf{q}^{p})$ is the unique solution to $ED(\mathbf{d}_1)$ and $\mathbf{r}_s^{p}$ is the unique solution to $FR(\mathbf{q}^{b},\mathbf{q}^p,\mathbf{d}_s)$. Under bids $(\boldsymbol\gamma^b,\boldsymbol\gamma^p)$, the problems $ED(\mathbf{d}_1)$ and $\hat{ED}(\mathbf{d}_1)$ are equivalent. Hence, $(\mathbf{q}^{b},\mathbf{q}^{p})$ is the unique solution to $\hat{ED}$, and by step 4, the production in the first time period is $(\mathbf{q}^b,\mathbf{q}^p)$. Under bids $(\boldsymbol\gamma^b,\boldsymbol\gamma^p)$, the problems $FR(\mathbf{q}^{b},\mathbf{q}^p,\mathbf{d}_s)$ and $\hat{FR}(\mathbf{q}^{b},\mathbf{q}^p,\mathbf{d}_s)$ are equivalent. Hence, $\mathbf{r}_s^{p}$ is the unique solution to $\hat{FR}(\mathbf{q}^{b},\mathbf{q}^p,\mathbf{d}_s)$, and by step 4, the recourse production is $\mathbf{r}_s^p$. Hence, the production schedule is $(\mathbf{q}^{b},\mathbf{q}^{p},\mathbf{r}^p)$.

It suffices to show that bids $(\boldsymbol\gamma^b,\boldsymbol\gamma^p)$ constitute an equilibrium. It is easy to check that $\boldsymbol\alpha^p = \boldsymbol\gamma^p$ and $\boldsymbol\alpha^b = \boldsymbol\gamma^b$ satisfy conditions~\eqref{eq:peaker-optimal-bid} and~\eqref{eq:baseload-optimal-bid} respectively for any prices $(\boldsymbol\pi^b,\boldsymbol\pi^p)$. Hence, simply choose $(\boldsymbol\pi^b,\boldsymbol\pi^p)$ based on equations~\eqref{eq:market-ed-kkt} and~\eqref{eq:market-fr-kkt}. This proves part (a) of the proposition.

\emph{Step 6: Showing that any bids $(\boldsymbol\alpha^b,\boldsymbol\alpha^p)$ that satisfy the characterizations in steps 1 to 3 give the same dispatch as that under bids $(\boldsymbol\gamma^b,\boldsymbol\gamma^p)$.} Suppose that $(\boldsymbol\alpha^b,\boldsymbol\alpha^p)$ satisfy the characterizations in step 4 with productions $(\mathbf{q}^b,\mathbf{q}_1^p,\ldots,\mathbf{q}_S^p)$, Lagrange multipliers $(\ubar{\boldsymbol\xi},\bar{\boldsymbol\xi},\ubar{\boldsymbol\nu},\bar{\boldsymbol\nu},\ubar{\boldsymbol\mu},\bar{\boldsymbol\mu},\boldsymbol\lambda)$, and prices $(\boldsymbol\pi^b,\boldsymbol\pi^p)$. We will construct $\ubar{\boldsymbol\xi}^\prime,\bar{\boldsymbol\xi}^\prime \in \mathbb{R}_+^N$ and $\ubar{\boldsymbol\nu}_1^{\prime},\bar{\boldsymbol\nu}_1^{\prime}\in\mathbb{R}_+^N$ such that:
\begin{subequations}
\label{eq:market-ed-alternate}
\begin{align}
\left(K c_n^{b\prime}(q_n^b), n\in N\right) + \bar{\boldsymbol\xi}^\prime - \ubar{\boldsymbol\xi}^\prime - K\boldsymbol\pi^b &= 0;
\label{eq:market-ed-baseload-foc-alternate}
\\
L^b(\mathbf{q}^b, \ubar{\boldsymbol\xi}^\prime, \bar{\boldsymbol\xi}^\prime) &= 0;
\label{eq:market-ed-baseload-comp-slack-alternate}
\\
\left(K c_n^{p\prime}(q_{1,n}^p), n\in N\right) + \bar{\boldsymbol\nu}_1^\prime - \ubar{\boldsymbol\nu}_1^\prime - K\boldsymbol\pi_1^p &= 0;
\label{eq:market-ed-peaker-foc-alternate}
\\
L^p(\mathbf{q}_1^{p},\ubar{\boldsymbol\nu}_1^\prime,\bar{\boldsymbol\nu}_1^\prime) &= 0,
\label{eq:market-ed-peaker-comp-slack-alternate}
\end{align}
\end{subequations}
and $\ubar{\boldsymbol\nu}_s^\prime,\bar{\boldsymbol\nu}_s^\prime \in \mathbb{R}_+^N$ for all $s\in S\setminus\{1\}$ such that:
\begin{subequations}
\label{eq:market-fr-alternate}
\begin{align}
\left(c_n^{p\prime}(q_{s,n}^p), n\in N\right) + \bar{\boldsymbol\nu}_s^\prime - \ubar{\boldsymbol\nu}_s^\prime - \boldsymbol\pi_s^p &= 0;
\label{eq:market-fr-peaker-foc-alternate}
\\
L^p(\mathbf{q}_s^{p},\ubar{\boldsymbol\nu}_s^\prime,\bar{\boldsymbol\nu}_s^\prime) &= 0,
\label{eq:market-fr-peaker-comp-slack-alternate}
\end{align}
\end{subequations}
which are the KKT conditions for~\eqref{eq:edhat} and~\eqref{eq:frhat} under bids $(\boldsymbol\gamma^b,\boldsymbol\gamma^p)$. Then, step 5 allows us to infer that the production schedule is an optimal solution to $SYSTEM$. Our construction is given by:
\begin{align*}
\ubar{\xi}_n^\prime 
&=
\begin{cases}
K\left(c_n^{b\prime}(\ubar{q}_n^b)-\pi_n^b\right), & \text{if} \ q_n^b = \ubar{q}_n^b;
\\ 
0, & \text{else},
\end{cases}
\\
\bar{\xi}_n^\prime
&=
\begin{cases}
K\left(\pi_n^b - c_n^{b\prime}(\bar{q}_n^b)\right), & \text{if} \ q_n^b = \bar{q}_n^b;
\\ 
0, & \text{else},
\end{cases}
\\
\ubar{\nu}_{1,n}^\prime 
&=
\begin{cases}
K\left(c_n^{p\prime}(\ubar{q}_{n}^p)-\pi_{1,n}^p\right), & \text{if} \ q_{1,n}^p = \ubar{q}_n^p;
\\ 
0, & \text{else},
\end{cases}
\\
\bar{\nu}_{1,n}^\prime
&=
\begin{cases}
K\left(\pi_{1,n}^p - c_n^{p\prime}(\bar{q}_n^p)\right), & \text{if} \ q_{1,n}^p = \bar{q}_n^p;
\\ 
0, & \text{else},
\end{cases}
\end{align*}
and:
\begin{align*}
\ubar{\nu}_{s,n}^\prime 
&=
\begin{cases}
c_n^{p\prime}(\ubar{q}_{n}^p)-\pi_{s,n}^p, & \text{if} \ q_{s,n}^p = \ubar{q}_n^p;
\\ 
0, & \text{else},
\end{cases}
\\
\bar{\nu}_{s,n}^\prime
&=
\begin{cases}
\pi_{s,n}^p - c_n^{p\prime}(\bar{q}_n^p), & \text{if} \ q_{s,n}^p = \bar{q}_n^p;
\\ 
0, & \text{else},
\end{cases}
\end{align*}
for all $s \in S\setminus\{1\}$. 

First, we show that $\ubar{\boldsymbol\xi}^\prime, \bar{\boldsymbol\xi}^{\prime}, \ubar{\boldsymbol\nu}_s^\prime,  \bar{\boldsymbol\nu}_s^\prime \geq 0$. Suppose $q_n^b = \ubar{q}_n^b$. Then, from~\eqref{eq:baseload-optimal-bid-a}, we infer that $c_n^{b\prime -1}(\pi_n^b) \leq \ubar{q}_n^b$, and since $c_n^b$ is strictly convex, we infer that $\pi_n^b \leq c_n^{b\prime}(\ubar{q}_n^b)$, and hence $\ubar{\xi}_n^\prime \geq 0$. Suppose $q_n^b = \bar{q}_n^b$. Then, from~\eqref{eq:baseload-optimal-bid-c}, we infer that $c_n^{b\prime -1}(\pi_n^b) \geq \bar{q}_n^b$, and since $c_n^b$ is strictly convex, we infer that $\pi_n^b \geq c_n^{b\prime}(\bar{q}_n^b)$, and hence $\bar{\xi}_n^\prime \geq 0$. By similar arguments, we infer that $\ubar{\nu}_{s,n}^\prime \geq 0$ and $\bar{\nu}_{s,n}^\prime \geq 0$.  

Second, we show that this construction satisfies~\eqref{eq:market-ed-alternate} and~\eqref{eq:market-fr-alternate}. It is easy to check that the complementary slackness conditions~\eqref{eq:market-ed-baseload-comp-slack-alternate},~\eqref{eq:market-ed-peaker-comp-slack-alternate},~\eqref{eq:market-fr-peaker-comp-slack-alternate} are satisfied. Suppose $\ubar{q}_n^b < c_n^{b\prime-1}(\pi_n^b) < \bar{q}_n^b$. From~\eqref{eq:baseload-optimal-bid-b}, we infer that $\alpha_n^b = \gamma_n^b$. From the fact that $q_n^b = \left[\alpha_n^b s_n^b(\pi_n^b)\right]_{\ubar{q}_n^b}^{\bar{q}_n^b} = c_n^{b\prime-1}(\pi_n^b)$, we infer that $\ubar{q}_n^b < q_n^b < \bar{q}_n^b$. From~\eqref{eq:market-ed-baseload-comp-slack}, we infer that $\ubar{\xi}_n = \bar{\xi}_n = 0$. Substituting into~\eqref{eq:market-ed-baseload-foc}, we infer that our construction satisfies~\eqref{eq:market-ed-baseload-foc-alternate}. Suppose $c_n^{b\prime-1}(\pi_n^b) \leq \ubar{q}_n^b$. From~\eqref{eq:baseload-optimal-bid-a}, we infer that $q_n^b = \ubar{q}_n^b$. Hence, our construction satisfies~\eqref{eq:market-ed-baseload-foc-alternate}. Suppose $\bar{q}_n^b \leq c_n^{b\prime-1}(\pi_n^b)$. From~\eqref{eq:baseload-optimal-bid-c}, we infer that $q_n^b = \bar{q}_n^b$. Hence, our construction satisfies~\eqref{eq:market-ed-baseload-foc-alternate}. Using similar arguments, we can infer that our construction satisfies~\eqref{eq:market-ed-peaker-foc-alternate} and~\eqref{eq:market-fr-peaker-foc-alternate}. This proves part (b) of the proposition.

\end{proof}









\end{document}